\documentclass[11pt]{article}
\usepackage{graphicx} 

\usepackage[margin=1in]{geometry}
\usepackage{setspace}

\usepackage{tikz}
\usetikzlibrary{automata,arrows,positioning,calc}

\usepackage[maxbibnames=99]{biblatex}
\makeatletter
\newcommand{\rom}[1]{\expandafter\@slowromancap\romannumeral #1@}
\makeatother

\usepackage{subcaption}
\usepackage{tikz}
\usepackage{amssymb}
\usetikzlibrary{automata, positioning, arrows}
\usepackage{amsthm}
\newtheorem{theorem}{Theorem}[section]
\newtheorem{corollary}[theorem]{Corollary} 
\newtheorem{lemma}[theorem]{Lemma} 
\theoremstyle{definition}
\newtheorem{definition}[theorem]{Definition} 

\newtheorem{prop}[theorem]{Proposition} 
\newtheorem{remark}[theorem]{Remark} 
\usepackage{amsmath}
\newtheorem{assumption}{Assumption}
\usepackage{geometry}
\usepackage{comment}   

\usepackage{color}

\newcommand{\BB}{\mathbb B}
\newcommand{\DD}{\mathbb D}
\newcommand{\EE}{\mathbb E}
\newcommand{\GG}{\mathbb G}
\newcommand{\LL}{\mathbb L}
\newcommand{\PP}{\mathbb P}
\newcommand{\RR}{\mathbb R}
\newcommand{\SQ}{\mathbb{S}\mathbb{Q}}
\newcommand{\TT}{\mathbb T}

\newcommand{\sE}{\mathcal E}
\newcommand{\sD}{\mathcal D}
\newcommand{\sG}{\mathcal G}
\newcommand{\sL}{\mathcal L}
\newcommand{\sT}{\mathcal T}

\newcommand{\NN}{\mathbb N}
\newcommand{\sF}{\mathcal F}
\newcommand{\Tau}{\mathcal T}

\title{A multi-player optimal stopping game based on a spider Brownian motion}
\author{David Hobson, Jingfei Liu}
\date{\today}

\begin{document}

\maketitle

\begin{abstract}
    This paper studies a multi-player non-zero-sum optimal stopping game driven by a spider Brownian motion. We study a specific, tractable family of games and the focus is on the existence and uniqueness of subgame-perfect Nash equilibria (SPNE).
        
    We begin by analyzing the two-player setting, in which the underlying process is a skew Brownian motion, and establish sufficient conditions for the existence of an SPNE. We then derive sufficient conditions for the uniqueness of the SPNE in the symmetric case, (where the skew Brownian motion reduces to a standard Brownian motion and both players have the same payoff structures). We also give examples which show that there can be multiple SPNE or no SPNE. We then extend the framework to the multi-player setting and provide sufficient conditions for the existence of an SPNE in this more general setting with a spider Brownian motion.
    
\end{abstract}
\section{Introduction}

The goal of this paper is to study a non-zero-sum optimal stopping game for two or more players, where the payoffs are functions of an underlying stochastic process. In order to make the problem as tractable as possible we choose a specific process (a skew Brownian motion, or in the multi-player case a spider Brownian motion) and specific payoff functions (chosen to exploit the scalings of the set-up). Then we exploit this tractability to determine existence and uniqueness of subgame-perfect Nash equilibria (SPNE) (at least amongst a class of threshold strategies for each player) in both the symmetric case where the skew Brownian motion is a standard Brownian motion and all players face the same payoff structures, and in the asymmetric case where the process is asymmetric or the agents face different payoffs.

At the heart of our study is an optimal stopping problem for Brownian motion which can be transformed into a stopping problem for a Brownian bridge. This problem was studied by, amongst others, Shepp \cite{Shepp1969} and Ekstr\"{o}m and Wanntorp \cite{EkstromWanntorp2009}. Ekstr\"{o}m and Wanntorp considered the problem of finding $\sup_{\sigma \in \sT_{[0,1]}} \EE[X_\sigma]$ where $\sT_{[0,1]}$ is the set of stopping times taking values in $[0,1]$ and $X = (X_t)_{t \geq 0}$ is a Brownian bridge and showed that the optimal stopping rule $\tau_*$ is of square-root threshold form: $\tau_* = \inf \{ u : X_u  \geq c^* \sqrt{1-u} \}$ where $c^*$ is the root of a certain equation.
Under the transformation $(X_u)_{0 \leq u \leq 1} = ( (1-u)B_{u/(1-u)})_{0 \leq u \leq 1}$ or equivalently $(B_t)_{t \geq 0} = ((1+t)X_{t/(1+t)})_{t \geq 0}$ we have that $B$ is a Brownian motion and the problem above is equivalent to finding
$\sup_{\tau \in \Tau} \EE[\frac{1}{1+\tau} B_\tau]$ where now $\Tau$ is the set of all stopping times (and on $\tau = \infty$ we take $\frac{1}{1+\tau} B_\tau \equiv 0$). The optimal stopping time is then given by $\tau^* = \inf \{ u \geq 0 : B_u \geq c^* \sqrt{1+u} \}$.

The problem of optimally stopping a Brownian bridge has been generalised in many directions. For example, Hobson and Liu \cite{HobsonLiu2026} studied the optimal stopping problem in which the underlying process is a Brownian excursion or a general $\alpha$-dimensional Bessel bridge, Görgens \cite{görgens2014optimalstoppingalphabrownianbridge} extended the problem to an $\alpha-$Brownian bridge, and De Angelis and Milazzo \cite{de2020optimal} studied the problem for the exponential of the Brownian bridge. Moreover, Baurdoux et al. \cite{Baurdoux_Chen_Surya_Yamazaki_2015} extended the
original optimal stopping problem to an optimal double stopping problem driven by a Brownian bridge; their goal is to find a pair of stopping times that maximizes the expected spread between the two payoffs.
Our aim is to extend the optimal stopping problem to an optimal stopping game.

We emphasise that there is little novelty in the optimal stopping problems that we study, and we deliberately choose forms of the games which are tractable and which have straightforward solutions. In this sense we are indifferent between the Brownian motion formulation or the bridge formulation since they are transformations of each other. (We choose the Brownian motion formulation purely because the skew- and spider Brownian motion are more widely studied than the skew- and spider Brownian bridge.)  
Instead, the novelty comes from the way these stopping problems are embedded in a Dynkin game. By having tractable stopping problems for the individual players we can focus on the game aspects and particularly the existence or otherwise of equilibria in games with two or more players.

Our base problem is a two player Dynkin game based on skew Brownain motion. Intuitively speaking, away from zero skew Brownian motion behaves like standard Brownian motion, but at zero it gets a push up or down with the effect that each excursion away from zero is positive with probability $p \in (0,1)$, and the signs of distinct excursions are independent. When $p=\frac{1}{2}$ we recover a Brownian motion.
Motivated by the analysis in \cite{EkstromWanntorp2009}
we choose a structure of the payoffs such that there is a penalty for stopping the game and such that given the opponent is using a square-root stopping rule then the optimal strategy is again of square-root form.
We look for a subgame-perfect Nash equilibrium (SPNE) within the class where both players use square-root strategies. In the symmetric case, where $p=\frac{1}{2}$ and the payoff functions are identical for the two players, then depending on the parameters of the payoffs, the game may admit a unique SPNE, multiple SPNE or no SPNE. In the asymmetric case, we prove the existence of a SPNE under certain conditions on the parameters by applying Miranda's theorem \cite{Miranda1940}.

Then we turn to a multi-player optimal stopping game based on a spider Brownian motion. 
There is extensive literature on two-player non-zero-sum games, but relatively little work has been done on multi-player optimal stopping games, and our goal is to provide a tractable Markovian example. The state space of a spider Brownian motion consists of $N >2$ copies of $(0,\infty)$ joined at a central vertex which we label as $0$. A spider or Walsh Brownian motion (for details see Walsh \cite{walsh_diffusion_1978}, Barlow et al. \cite{barlow2006walsh}) 
behaves like Brownian motion whilst on a given ray until it hits $0$, at which point it chooses a new ray along which it again acts as a Brownian motion. Each time it hits the vertex it chooses Ray $i$ for $i \in \{1, \ldots , N \}$, with probability $p_i>0$ where $(p_i)_{ 1  \leq i \leq N }$ is a probability vector, i.e. $\sum_{i=1}^N p_i = 1$. The choices of ray for different excursions away from $0$ are independent. 

We consider a $N$-player optimal stopping game based on a spider Brownian motion. Each player is identified with a ray, and players may only stop when the process is in their ray (or at $0$).
There is a penalty for being the player who stops the game. 
From the scalings of the problem (and the intuition developed in the skew Brownian motion case) we expect each player to follow a square-root threshold stopping rule. Our goal is to find SPNE within this class of strategies. Our main result is a set of sufficient conditions for the existence of SPNE.

Zero-sum optimal stopping games, often referred to as Dynkin games, were first introduced by Dynkin \cite{Dynkin1969_OptimalStoppingGame} as a natural extension of optimal stopping problems. Neveu extended the game to a discrete-time setup \cite{Neveu1975}, and Bismut studied the game in a continuous-time setting \cite{Bismut1977Dynkin}. One of the main focuses of studying Dynkin games is establishing the existence of an equilibrium.  Ekstr\"{o}m and Peskir \cite{EkstromPeskir2008} studied an optimal stopping game for a Markov process, and find sufficient conditions for the existence of a Stackelberg equilibrium and a Nash equilibrium. Alvarez \cite{Alvarez2008SolvableStoppingGames} considered a class of Dynkin games driven by a one-dimensional diffusion process, and provided general conditions under which the saddle point and equilibria exist, as well as sufficient conditions under which the equilibrium strategy takes the form of threshold stopping times. 
The study of non-zero-sum games began with Bensoussan and Friedman~\cite{BensoussanFriedman1977} who worked with diffusion processes and proved that a Nash equilibrium exists if there is a solution to a system of quasi-variational inequalities. 
Morimoto \cite{Morimoto1986} studied the same problem in a discrete setting, and found sufficient conditions for the existence of Nash equilibrium using the martingale method and a fixed point theorem for monotone mappings. Since then, the research on non-zero-sum optimal stopping games has concentrated on the existence and characterization of Nash equilibria and SPNE. 

 Hamad\`{e}ne and Zhang~\cite{HamadeneZhang2010} consider a two-player non-zero sum Dynkin game and under some structural conditions prove existence of a Nash equilibrium using backward stochastic differential equations and Snell envelopes. Laraki and Solan~\cite{LarakiSolan2013} prove existence of a $\epsilon$-optimal Nash equilibrium in a wider class of set-ups using randomised stopping times. Hamad\`{e}ne and Hassani~\cite{HamadeneHassani2014} generalise the existence result to the multi-player setting. It should be noted that Laraki et al~\cite{LarakiSolanVieille2005} give a counterexample to show that in a multi-player setting, we cannot expect existence of a Nash equilibrium in general, even with randomised strategies.

All the papers mentioned in the last paragraph work in a setting with abstract payoffs. A second strand of literature considers problems in which
the payoffs are functions of an underling process.
Peskir \cite{Peskir2009} studied a non-zero-sum stopping game driven by a right-continuous strong Markov process and established that a Nash equilibrium exists if and only if the value function admits a semi-harmonic characterization, meaning that the value function coincides with the smallest superharmonic function and the largest subharmonic function lying between the gain and loss function.
Attard \cite{Attard2015} worked on a right-continuous, quasi-left-continuous strong Markov process, and constructed a Nash equilibrium by deriving a partial superharmonic characterization for the value function for each player.
The same author later also worked on non-zero-sum optimal stopping games for Brownian motion \cite{Attard2017}, in which the value functions were derived by transforming the game into a free-boundary problem. A more recent work by De Angelis et al. \cite{DeAngelisFerrariMoriarty2018} established the existence (and uniqueness under extra conditions) of a Nash equilibrium in a non-zero-sum game driven by a one-dimensional diffusion. They further studied the structure of equilibria, and provided sufficient conditions for when the Nash equilibrium takes a simple threshold form, in which case the threshold can be determined by solving a system of algebraic equations. 
Typically these papers assume that the game involves two players. As far as we are aware we are the first authors to consider a Markovian-based, continuous-time Dynkin game with more than two players.

 In the Markovian setting (and in this paper) it is possible to consider a different notion of Nash equilibria, namely a subgame-perfect Nash equilibrium (SPNE), in which we consider starting the game at any possible point, and require equilibrium strategies to define a Nash equilibrium from a general starting point. One of the results we find is that with this more restricted definition of Nash equilibria the existence results of Hamed\'ene and Zhang no longer apply, see Remark~\ref{rem:counterexampleHamedene}.

The structure of the paper is as follows. In Section 2, we begin with a formal statement of the two-player game based on a skew Brownian motion. Then, fixing Player 2's strategy, we reduce the optimal stopping game to an optimal stopping problem for Player 1, which we solve to give an explicit value function for Player 1. This result can be used to find SPNE for the game, and we do so first in the asymmetric case where we focus on the existence and second in the symmetric case where we look for symmetric equilibria, and also consider the uniqueness. In the symmetric case, depending on the parameters of the problem there may be no symmetric equilibrium, a unique equlibrium or multiple equilibria, but never more than three. 
In Section 3, we study the $N-$player game formulation, and present sufficient conditions for the existence of SPNE.

\section{The Two-player game with a skew Brownian motion}
\subsection{Game Formulation}

We work on a complete filtered probability space $(\Omega, \mathcal{F},\mathbb{P},\mathbb{F}=(F_{t})_{t\geq0})$ which satisfies the usual conditions. 
On this space
we have 
a skew Brownian motion $X=\{X_{t}\}_{0 \leq t\leq 1}$.

\begin{definition}
    A skew Brownian motion is a real-valued, continuous stochastic process $X=\{X_{t}\}_{0\leq t < 1}$ which satisfies the SDE
\begin{equation}\label{SBMbridge_SDE}
    dX_{t}=dB_{t}+(2p-1)d\Tilde{L}_{t}^{X,0}, \hspace{+1CM} 0\leq t < 1,
    \end{equation}
    where $(\Tilde{L}_{t}^{X,0})_{t \geq 0}$ is the symmetric local time of $X$ at zero,
    subject to the initial condition $X_{0}=0$.  
\end{definition}

We consider $X$ to be the underlying stochastic process in a non-zero-sum, optimal stopping (Dynkin) game. There are two players, labeled Player 1 and Player 2 both of whom aim to maximize their expected payoffs. Each player chooses a stopping time: $\tau$ for Player 1, $\sigma$ for Player 2. If $\tau<\sigma$, meaning that Player 1 stops the game first, then Player 1 receives $\GG_{1}= \GG(\tau,X_\tau)$ and Player 2 receives $\LL_{2}(\tau,X_\tau)$.
Conversely, if $\sigma<\tau$, Player 2 receives $\GG_{2}(\sigma,X_\sigma)$, while Player 1 receives $\LL_{1}(\sigma,X_\sigma)$. We assume that Player 1 is only allowed to stop the game when $X \geq 0$ and Player 2 is only allowed to stop the game when $X \leq 0$, so that if $\tau = \sigma < \infty$ then $X_\tau=0$
If $\tau=\sigma$, meaning that both players stop the game at the same time, then Player 1 receives $\GG_1$ and Player 2 receives $\GG_2$, the same payoffs as if they had been the sole player who stopped the game. 
Here, the functions
$\GG_{1},\LL_{2}: [0,\infty]\times [0,\infty) \to \RR$ and $\GG_{2},\LL_{1}: [0,\infty]\times (-\infty,0] \to \RR$ are assumed to be continuous with $\GG_i(\infty,x)= \LL_i(\infty,x)=0$. 
Putting this together, the expected payoffs $R_1$ to Player 1 and $R_2$ to Player 2 (if the two players use stopping times $\tau$ and $\sigma$ respectively) are
\begin{eqnarray*}
     R_{1}(\tau,\sigma) & = &\mathbb{E}\left[\GG_{1}(\tau, X_{\tau})\mathbb{I}_{\{\tau \leq \sigma<\infty\}}+ \LL_{1}(\sigma,X_{\sigma})\mathbb{I}_{\{\sigma<\tau<\infty\}} \right],  \\
     R_{2}(\tau,\sigma) & = &\mathbb{E}\left[\GG_{2}(\sigma,X_{\sigma})\mathbb{I}_{\{\sigma \leq \tau<\infty\}}+ \LL_{2}(\tau,X_{\tau})\mathbb{I}_{\{\tau<\sigma<\infty\}}\right].
\end{eqnarray*}

Let $\sT$ be the set of stopping times taking values in $[0,\infty]$ (and for later use, let $\sT_t = \sT_{[t,\infty]}$ be the set of stopping times taking values in $[t,\infty]$). We are interested in finding subgame-perfect Nash equilibria (SPNE) for the game.

\begin{definition}\label{Nash_Equilibrium}
   We say that $(\tau^{*},\sigma^{*}) \in \sT \times \sT$ is a Nash equilibrium (for the game with expected payoffs $(\GG_i, \LL_i)_{i = 1,2}$) if 
    \begin{equation*}
        \begin{split}
            R_{1}(\tau^{*},\sigma^{*})=\sup_{\tau \in \sT} R_{1}(\tau,\sigma^{*}),  \\
            R_{2}(\tau^{*},\sigma^{*})=\sup_{\sigma \in \sT} R_{2}(\tau^{*},\sigma).
        \end{split}
    \end{equation*}
\end{definition}

Underpinning our Dynkin game is an optimal stopping problem for a Brownian bridge which was solved by Ekstr\"{o}m and Wanntorp \cite{EkstromWanntorp2009}. Consider the (single agent) problem of finding $\sup_{\sigma \in \sT_{[0,1]}} \EE[Y_\sigma]$ where $Y$ is a Brownian bridge. 
Ekstr\"{o}m and Wanntorp \cite{EkstromWanntorp2009} exploited scalings of the Brownian bridge to find the value function for the problem for an arbitrary starting point $(t,x)$ and showed that the optimal stopping rule $\tau^*$ is of square-root threshold form: $\sigma^* = \inf \{ u \in [0,1] : Y_u \geq c^* \sqrt{1-u} \}$ where $c^* \in (0,\infty)$ is a constant which they determined. We want to consider similar scalings in our game. Making the transformation $B_t = (1+t)Y_{t/(1+t)}$ for $t \geq 0$ we find that $\sup_{\sigma \in \sT_{[0,1]}} \EE[Y_\sigma]= \sup_{\tau \in \sT} \EE[\frac{1}{(1+\tau)}B_\tau]$ and that the optimal stopping rule is $\tau^* = \inf \{ t\geq 0: B_t \geq c^* \sqrt{1+t} \}$. This problem becomes the building block of our study.

To this end, we suppose $\GG_1,\LL_2:[0,\infty] \times [0,\infty) \to \RR$ are of the form $\GG_1(t,x) = \frac{1}{\sqrt{1+t}} \sG_1( \frac{x}{\sqrt{1+t}} )$ and $\LL_2(t,x) = \frac{1}{\sqrt{1+t}} \sL_2( \frac{x}{\sqrt{1+t}} )$. Similarly we suppose $\GG_2,\LL_1:[0,\infty] \times (-\infty,0] \to \RR$ are of the form $\GG_2(t,x) = \frac{1}{\sqrt{1+t}} \sG_2( \frac{x}{\sqrt{1+t}} )$ and $\LL_1(t,x) = \frac{1}{\sqrt{1+t}} \sL_1( \frac{x}{\sqrt{1+t}} )$. Our main example is the case where (for $x>0$)
$\GG_1(t,x) = \frac{x}{1+t} - \alpha_1 \frac{1}{\sqrt{1+t}}$ and $\LL_2(t,x) = - \frac{\theta_2x}{1+t}$ and (for $x<0$) $\GG_2(t,x) = -\frac{x}{1+t} - \alpha_2 \frac{1}{\sqrt{1+t}}$ and $\LL_1(t,x) = \frac{\theta_1x}{1+t}$.
Here $\alpha_1, \alpha_2, \theta_1, \theta_2>0$. 
Then $\sG_1(z) = z - \alpha_1$, $\sG_2(z) = - z -\alpha_2$, $\sL_1(z) = \theta_1 z$ and $\sL_2(z)= - \theta_2 z$.

The strategy for each player is a stopping rule. Within the class of stopping times is the class $\sT^{\TT}$ of threshold strategies and within the class of threshold strategies is the class $\sT^{\SQ}$ of threshold strategies of square-root form. A stopping time $\tau$ (respectively, $\sigma$) is a threshold strategy if $\tau = \tau^B= \inf\{ t : X_t \geq B(t) \}$ (respectively, $\sigma = \inf\{ t : X_t \leq  -D(t) \}$) where $B,D:[0,\infty) \to [0,\infty]$  
are measurable functions (here we use the fact that Player 1 is only able to stop when $X_t \geq 0$ and Player 2 is only able to stop when $X_t \leq 0$). A threshold strategy $\tau$ (respectively, $\sigma$) is a threshold strategy of square root form if $\tau = \tau^b = \inf\{ t : X_t \geq b \sqrt{1+t}  \}$ (respectively $\sigma = \sigma^d = \inf\{ t : X_t \leq - d \sqrt{1+t}  \}$) where $b,d \geq 0$.

A standard approach in control problems such as optimal stopping problems and optimal stopping games is to consider the problem started from a general intermediate point, in our case $X_t = x$, for $t \in [0,\infty)$ and $x \in \RR$. We let $\PP_{t,x}$ (and associated expectation operator $\EE_{t,x})$ denote the conditional probabilities based on $X_t=x$. Further, 
let $\tau^B \in \sT^\TT$ be a threshold stopping time of the form $\tau^B = \inf \{ t \geq 0 : X_t \geq B(t) \}$. Then for $t \geq 0$ define $\tau^B_t = \inf \{ s \geq t : X_s \geq B(s) \}$. Similarly, given a threshold strategy $\sigma^D = \inf \{ t \geq 0 : X_t \leq -D(t) \}$ for Player 2, define $\sigma^D_t = \inf \{ s \geq t : X_s \leq - D(s) \}$. 

Then we can write
\begin{eqnarray*} 
R_1(x,t;\tau^B,\sigma^D) & = & \EE_{t,x} \left[ \GG_1(\tau^B_t,X_{\tau^B_t}) 
\mathbb{I}_{\{\tau^B_t \leq \sigma^D_t \}}+ \LL_{1}(\sigma^D_t,X_{\sigma^D_t})\mathbb{I}_{\{\sigma^D_t<\tau^B_t\}}
\right] , \\
R_2(x,t;\tau^B,\sigma^D) & = & \EE_{t,x} \left[ \GG_2(\tau^B_t,X_{\tau^B_t}) 
\mathbb{I}_{\{\sigma^D_t \leq \tau^B_t \}}+ \LL_{2}(\sigma^D_t,X_{\sigma^D_t})\mathbb{I}_{\{\tau^D_t < \sigma^D_t \}}
\right]  .
\end{eqnarray*}

\begin{definition}\label{spNash_Equilibrium}
   We say that $(\tau^{*},\sigma^{*}) \in \sT^\TT \times \sT^\TT$ is a subgame-perfect Nash equilibrium (SPNE) of threshold-type if for all $t\in[0,\infty)$ and $x=X_{t}\in\mathbb{R}$, the stopping times $\tau^{*}_t$ and $\sigma^{*}_{t}$ satisfy  
    \begin{equation*}
        \begin{split}
            R_{1}(t,x;\tau_{t}^{*},\sigma_{t}^{*})=\sup_{\rho \in \sT_t} R_{1}(t,x;\rho,\sigma^{*}_{t}),  \\
            R_{2}(t,x;\tau^{*}_{t},\sigma^{*}_{t})=\sup_{\rho \in \sT_t} R_{2}(t,x;\tau^{*}_{t},\rho).
        \end{split}
    \end{equation*}
\end{definition}

\subsection{The Optimal Stopping Problem for Player 1}

In this section, we fix Player 2's strategy to be $\sigma = \sigma^d$ for $d \in [0,\infty)$ where $\sigma^d :=\inf\{s\geq 0: X_{s}\leq - d \sqrt{1+s} \}$. We find the value function $\overline{V}_{1,\sigma^d}$ and the optimal stopping rule for Player 1 as a function of a general starting point. In particular, $\overline{V}_{1,\sigma^d} : [0,\infty) \times \RR \to \RR$ is given by
\begin{equation}\label{Valuefunc_P1}
   \overline{V}_{1,\sigma^d}(t,x) = \sup_{\tau \in \sT_t}  
 \EE_{t,x} \left[ \GG_1(\tau,X_{\tau}) 
\mathbb{I}_{\{\tau \leq \sigma^d_t<\infty \}}+ \LL_{1}(\sigma^d_t,X_{\sigma^d_t})\mathbb{I}_{\{\sigma^d_t<\tau<\infty\}}
\right] . 
\end{equation}
Since we are only interested in Player 1's problem and since $d$ and $\sigma^d$ are fixed, in this section we omit the subscripts and write $\overline{V}$ as shorthand for $\overline{V}_{1,\sigma^d}$.

We begin by developing a candidate solution which we denote by $$V^{\text{cand}} =V= (V(t,x))_{(t,x) \in [0,\infty) \times \RR}.$$ Later we show that $V= \overline{V}$.

Since $V$ is a candidate solution 
we may assume that $V$ is appropriately smooth 
and write $\dot{}$ for a time derivative and ${}'$ for a space derivative. For now suppose that Player 1 follows a threshold strategy with threshold $B \geq 0$, so that $\tau = \tau^B = \inf \{ u : X_u \geq B(u) \}$. Later we will argue that it is optimal for Player 1 to use a threshold strategy of square-root form, and we will prove the optimality over threshold strategies extends to optimality over all stopping times.

Let $D(t)= d \sqrt{1+t}$. Then
the region $\{(t,x):[0,\infty)\times(-D(t),B(t))\}$ is the continuation region, and $\{(t,x):[0,\infty)\times[B(t),\infty)\}$ is the stopping region. If $X_t \leq - D(t)$ then the game is stopped by Player 2 and this stopping is exogenous from the perspective of Player 1. In the stopping region for Player 1 we expect that $V(t,x)= \GG_1(t,x)$ and in the region where Player 2 stops we have $V(t,x)= \LL_1(t,x)$. We seek to find $V$ in the continuation region, and to find the optimal threshold, i.e. the optimal function $B:[0,\infty) \to \RR$.

Suppose that the game has not been stopped by time $t_0$ and that $X_{t_0}= x_0 \in (-D(t_0), B(t_0))$. We take $(t_0,X_{t_0}=x_0)$ to be our initial condition. Then, 
using the It\^o-Tanaka formula with symmetric local times representation (see Revuz and Yor~\cite[Chapter~VI, p.~234]{revuz_yor_1999}), we have 
\begin{equation*}
   V(t, X_{t})=V(t_0,x_0)+\int_{t_0}^{t}\Dot{V}(s,X_{s})ds + \int_{t_0}^{t}V'(s, X_{s})dX_{s}+\frac{1}{2}\int \Tilde{L}^{X,a,t_0}_{t}V''(da),
\end{equation*}
where $\tilde{L}^a_t=\tilde{L}^{X,a,t_0}_{t}$ is the symmetric local time of $X$ at $a$ between times $t_0$ and $t$. Applying this result for $X$ a skew Brownian motion and using the Occupation Times formula (see Revuz and Yor~\cite[Chapter~VI, p.~224]{revuz_yor_1999}), along with $V'(t,0)=\frac{1}{2}\left[V'(t,0+)+V'(t,0-)\right]$ and $\Delta V'(t,0)= V'(t,0+)-V'(t,0-)$ we have

\[  dV(t, X_{t})= 
\Dot{V}(t,X_{t}) + V'(t, X_{t}) dX_{t}+\frac{1}{2} V''(t, X_{t})d[X]_{t}
+\frac{1}{2}\Delta V'(t, 0) d\Tilde{L}^0_{t}.\]

Substituting \eqref{SBMbridge_SDE} into the above equation, and using that $d\Tilde{L}_{t}^{0}$ is only supported on the set $\{s\in [t_0,t]: X_{s}=0\}$ we find
\begin{equation*}
  \begin{split}
dV(t,X_{t})=&\Dot{V}(t,X_{t})dt+V'(t, X_{t})\left[dB_{t}+(2p-1)d\Tilde{L}_{t}^{0}\right]\\
      &\hspace{+4cm}+\frac{1}{2}V''(t, X_{t})dt+\frac{1}{2}\left[V'(t,0+)-V'(t,0-)\right]d\Tilde{L}_{t}^{0}\\
      =&\left[\Dot{V}(t,X_{t})+\frac{1}{2}V''(t, X_{t})\right]dt+V'(t, X_{t})dB_{t}\\
      &\hspace{+4CM}+\left\{V'(t, 0)(2p-1)+\frac{1}{2}\left[ V'(t,0+)-V'(t,0-)\right]\right\}d\Tilde{L}_{t}^{0}\\
      =&\left[\Dot{V}(t,X_{t})+\frac{1}{2}V''(t, X_{t})\right]dt+V'(t, X_{t})dB_{t}
      +\left[p V'(t, 0+) - (1-p) V'(t,0-) \right]d\Tilde{L}_{t}^{0}.
  \end{split}
\end{equation*}
We expect $V=V(t, X_{t})_{t \geq 0}$ to evolve as a martingale in the continuation region. It follows that the coefficients premultiplying the $dt$ term and $d\Tilde{L}_{t}^{0}$ term should be 0. Thus, we expect the conditions
\begin{align}
   \Dot{V}(t,x)+\frac{1}{2}V''(t,x) &= 0, \hspace{+1cm} x \in (-D(t),0) \cup (0, B(t));  \label{Valuefunc_PDE}\\
   pV'(t,0+)&=(1-p)V'(t,0-).\label{ConditionLocalTime_original}
\end{align}
Moreover, value-matching conditions at the boundaries give us, 
\begin{align}
     V(t,x)&= \mathbb{G}_{1}(t,x), \hspace{+0.5cm} x= B(t), \label{ValueMatching_original_1}\\
   V(t,x)&= \mathbb{L}_{1}(t,x), \hspace{+0.5cm} x=-D(t). \label{ValueMatching_original_2}
\end{align}
Recalling that $\sigma^d = \inf \{ u : X_u \leq -d \sqrt{1+u} \}$, 
we make the ansatz that the optimal stopping rule for Player 1 is a threshold strategy of square-root form $\tau = \tau^b = \inf \{ u : X_u \geq b \sqrt{1+u} \}$, and that when $x>0$, 
\begin{equation*}
    V(t,x)=\frac{1}{\sqrt{1+t}}f_{+}\left(\frac{x}{\sqrt{1+t}}\right),
\end{equation*}
and when $x<0$, $V(t,x)=\frac{1}{\sqrt{1+t}}f_{-}\left(\frac{x}{\sqrt{1+t}}\right)$ where $f_+:\mathbb{R}_{+}\to \RR$ and $f_-:\mathbb{R}_{-}\to \RR$ are functions to be determined.   
In this way, setting $z=\frac{x}{\sqrt{1+t}}$, \eqref{Valuefunc_PDE} and \eqref{ConditionLocalTime_original} are transformed into 
\begin{align}
   f''_{+}(z)-zf'_{+}(z)-f_{+}(z)&=0, \hspace{+1cm} 0<z<b,\label{ValueFunc_ODE_positive}\\
   f''_{-}(z)-zf'_{-}(z)-f_{-}(z)&=0, \hspace{+1cm} -d<z<0,\label{ValueFunc_ODE_negative}\\
   pf'_{+}(0+)&=(1-p)f'_{-}(0-). \label{ValueFunc_ODE_thirdcondition}
\end{align}
If $d=0$ then \eqref{ValueFunc_ODE_negative} is vacuous and \eqref{ValueFunc_ODE_thirdcondition}
is replaced by $f_+(0)= \sL_1(0)$.
Recalling the relationships between $\GG_1$ and $\sG_1$, and $\LL_1$ and $\sL_1$,
we also have that when $z\geq b$, $f_+(z) = \sG_1(z)$ and when $z\leq -d$, $f_-(z) = \sL_1(z)$.

We want to construct the value function for Player 1 under stopping rule $\tau^b$ and later to maximise over $b$ to construct a candidate optimal stopping rule.

The equation $f''(z)-zf'(z)-f(z)=0$ has fundamental solutions
\begin{equation}\label{Fundamental_solutions}
    \phi(z)=e^{\frac{z^{2}}{2}} \ \text{ and } \ \psi(z)= e^{\frac{z^{2}}{2}}\int_{0}^{z}e^{-\frac{u^{2}}{2}}du,
\end{equation}
Here $\phi$ and $\psi$ are chosen such that
\begin{equation}
\label{eqn:phipsiproperties}
    \begin{split}
        &\phi(0)=1, \ \phi'(0)=0,  \\
        &\psi(0)=0, \ \psi'(0)=1 .
    \end{split}
\end{equation}
We find that 
\begin{equation}
\label{eqn:phipsiderivaties}
\ \phi'(z)=z\phi(z), \hspace{20mm} \psi'(z)=z\psi(z)+1. 
\end{equation}

Then, \eqref{ValueFunc_ODE_positive}---\eqref{ValueFunc_ODE_negative} have solutions
\begin{equation}\label{Expression_f_2players}
    \begin{cases}
      f_{+}(z)  =  A_{+}(b,d) \phi(z)+B_{+}(b,d)\psi(z), & \hspace{+1cm} 0<z<b\\
      f_{-}(z) = A_{-}(b,d) \phi(z)+B_{-}(b,d)\psi(z), & \hspace{+1cm} -d<z<0,
    \end{cases}
\end{equation}
for constants $A_{\pm}(b,d)$ and $B_{\pm}(b,d)$
to be determined by the fact that $f_+(0+) = f_-(0-)$, the value-matching conditions at the boundaries and \eqref{ValueFunc_ODE_thirdcondition}.

Since we have $\lim_{z\downarrow 0}f_{+}(z)=\lim_{z\uparrow 0}f_{-}(z)$, we must have $A_{+}(b,d)=A_{-}(b,d)=:A(b,d)$: then \eqref{ValueFunc_ODE_thirdcondition} with \eqref{eqn:phipsiproperties} yields 
\begin{equation}
    \label{eq:Bderivatives}
    pB_{+}(b,d)-(1-p)B_{-}(b,d)=0.
\end{equation} 
Moreover, after the transformation, the value matching conditions are satisfied at $b$ and $-d$, and thus \eqref{ValueMatching_original_1} and \eqref{ValueMatching_original_2} are equivalent to 
\begin{align}
    A(b,d) \phi(b)+B_{+}(b,d)\psi(b) &=\mathcal{G}_{1}(b) \label{ValueMatching_1},\\
    A(b,d) \phi(-d)+B_{-}(b,d)\psi(-d) &=\mathcal{L}_{1}(-d)\label{ValueMatching_2}.
\end{align}

First, consider the case where Player 2 chooses the boundary $d>0$. Using \eqref{eq:Bderivatives} to eliminate $B_+$ and $B_-$ from 
\eqref{ValueMatching_1} and \eqref{ValueMatching_2} we
obtain,
\begin{eqnarray}
\label{Expression_coefficients_2playersA}
       A (b,d) &= & \frac{(1-p)\psi(b)\mathcal{L}_{1}(-d)+p\psi(d)\mathcal{G}_{1}(b)}{(1-p)\psi(b)\phi(d)+p\phi(b)\psi(d)} ,\\
\label{Expression_coefficients_2playersB+}     
B_{+}(b,d)& = &(1-p)\left[\frac{\mathcal{G}_{1}(b)\phi(d)-\phi(b)\mathcal{L}_{1}(-d)}{(1-p)\psi(b)\phi(d)+p\phi(b)\psi(d)}\right],\\
\label{Expression_coefficients_2playersB-}
    B_{-}(b,d) &=&p\left[\frac{\mathcal{G}_{1}(b)\phi(d)-\phi(b)\mathcal{L}_{1}(-d)}{(1-p)\psi(b)\phi(d)+p\phi(b)\psi(d)}\right],
\end{eqnarray}
where we use that $\phi(-d)=\phi(d)$ and $\psi(-d)=-\psi(d)$.
Moreover, from \eqref{ValueMatching_2} we have that $B_-(b,d)=-\frac{\mathcal{L}_{1}(-d)}{\psi(d)}+\frac{\phi(d)}{\psi(d)}A(b,d)$, and then we have 
\begin{equation}
\label{eq:B+def}
    B_+(b,d)=\frac{1-p}{p}\left[-\frac{\mathcal{L}_{1}(-d)}{\psi(d)}+\frac{\phi(d)}{\psi(d)}A(b,d)\right].
\end{equation}

Now consider the case where Player 2 chooses the boundary $d=0$. \eqref{ValueMatching_1} and \eqref{ValueMatching_2} become
\begin{align}
   A(b,0) \phi(b)+B_{+}(b,0)\psi(b) &=\mathcal{G}_{1}(b), \label{ValueMatching_0_1}\\
    A(b,0)  &=\mathcal{L}_{1}(0)\label{ValueMatching_0_2}, 
\end{align}
and we obtain,
\begin{equation}\label{ExpressionAB_0}
A(b,0)  =\mathcal{L}_{1}(0), \hspace{+1cm}  B_{+}(b,0)=\frac{\mathcal{G}_{1}(b)-\mathcal{L}_{1}(0)\phi(b)}{\psi(b)}.  
\end{equation}
Note the expression for $B_+(b,0)$ given in \eqref{ExpressionAB_0} coincides with the general expression for $B_+(b,d)$ given in \eqref{Expression_coefficients_2playersB+} when $d=0$. Therefore, for the following definition, we combine the above two cases together. 

Since $d \in [0,\infty)$ is taken fixed in this section, we define 
$\Gamma_1 = \Gamma_{1,d}=\Gamma_{1,d, p, \alpha, \theta} : [0,\infty) \to \RR$ by
\begin{equation}\label{Gamma_expression_}
  \Gamma_{1}(b) =  B_+(b,d) = (1-p)\left[\frac{\mathcal{G}_{1}(b)\phi(d)-\phi(b)\mathcal{L}_{1}(-d)}{(1-p)\psi(b)\phi(d)+p\phi(b)\psi(d)}\right]
\end{equation}

For $d>0$ define $f = f^{b,d}: \mathbb{R} \to \RR$ by
\begin{equation}
\label{eq:fdef}
f^{b,d}(z) = \begin{cases}
    A(b,d)  \quad \quad \quad & z=0; \\
    f_+(z)   &  0 < z < b; \\
    f_-(z)   &  -d < z < 0; \\
    \sG_1(z)   &  z \geq b; \\
    \sL_1(z)   &  z \leq -d; \\
\end{cases} \end{equation}
For $d=0$, define $f = f^{b,0}: \mathbb{R} \to \RR$ by
\begin{equation}
\label{eq:fdef_0}
f^{b,0}(z) = \begin{cases}
    f_+(z)   &  0 < z < b; \\ 
    \sG_1(z)   &  z \geq b; \\
    \sL_1(z)   &  z \leq 0; \\
\end{cases} \end{equation}
Thus \eqref{eq:fdef_0} is a special case of \eqref{eq:fdef} in which the third line of \eqref{eq:fdef} is redundant and the first and last lines agree at $z=0$. The next assumption is a condition on the payoffs $\mathcal{G}_1(\cdot)$ and $\mathcal{L}_1(\cdot)$.

\begin{assumption}
\label{ass:Gamma}
 For each $d \geq 0$,
$\Gamma_{1,d}$ is maximised at $b^*=b^*(d) = b^*(d,p,\alpha,\theta) \in [0,\infty)$ and $\Gamma_{1,d}$ is non-increasing on $[b^*,\infty)$.
\end{assumption}

\begin{prop}
\label{prop:maximiser}
Fix $d\in[0,\infty)$. Suppose that $\Gamma_{1,d}$ satisfies Assumption~\ref{ass:Gamma}.

    Then $f^{b,d}$ (resp. $f^{b,0}$) given in \eqref{eq:fdef} (resp.\eqref{eq:fdef_0}), attains its maximum, uniformly in $z \in \mathbb{R}$, at $b^*=b^*(d)$ where
    $b^{*}$ is the maximizer of $\Gamma_{1,d}$.     
\end{prop}

\begin{proof}
    First consider the case $d>0$ where $d$ is fixed. Suppose $b^*$ is the maximiser of $\Gamma_{1}$. When $z\leq-d$ we have that $f^{b,d}(z)=\mathcal{L}_{1}(z)$ which does not depend on $b$ and hence $b^*$ maximizes $f^{b,d}(z)$ when $z\leq -d$.
    From the definition of $\Gamma_{1,d}$ given in \eqref{Gamma_expression_} and the relation between $B_+$ and $B_-$ given in \eqref{eq:Bderivatives}, we know immediately that if $b^*$ maximises $\Gamma_{1,d}$, it also maximises $B_+(b,d)$ and $B_-(b,d)$. Moreover, adding the value-matching condition given in \eqref{ValueMatching_2}, we have $A(b,d) = \frac{1}{\phi(d)}(\sL_1(-d) + \frac{p}{1-p} \psi(d) \Gamma_{1,d}(b))$, thus $b^*$ also maximises $A(b,d)$. Thus, $b^*$ maximizes $f^{b,d}(z)$ when $z=0$. 

    Now suppose $-d<z<0$. Using \eqref{ValueMatching_2} we have that
    $A(b,d) = \frac{\sL_1(-d)}{\phi(d)} + \frac{\psi(d)}{\phi(d)}B_-(b,d)$ and $B_-(b,d)=\frac{p}{1-p}\Gamma_{1,d}(b)$, so that for $-d<z<0$,
    \begin{equation*}
        f^{b,d}(z)=\frac{\sL_1(-d)}{\phi(d)} \phi(z)+\left(\psi(z)+\frac{\psi(d)}{\phi(d)}\phi(z)\right)\frac{p}{1-p}\Gamma_{1,d}(b).
    \end{equation*}
    Since $\frac{\psi(z)}{\phi(z)}$ is increasing in $z$ and $-d<z<0$, we have that $\frac{\psi(z)}{\phi(z)}>\frac{\psi(-d)}{\phi(-d)}=-\frac{\psi(d)}{\phi(d)}$, and thus the coefficient in front of $\Gamma_{1,d}(b)$ is positive and $f^{b,d}(z)$ is maximised by the maximiser $b^*=b^*(d)$ of $\Gamma_{1,d}$.

    Now suppose $z>0$. From \eqref{eq:fdef}, when $0<z<b$ 
    \begin{equation*}
      f^{b,d}(z)=
      \frac{\sL_1(-d)}{\phi(d)} \phi(z) + \left( \psi(z) +  \frac{\psi(d)}{\phi(d)}\frac{p}{1-p} \phi(z) \right) \Gamma_{1,d}(b).  
    \end{equation*}
     Note that $\frac{\psi(d)}{\phi(d)}\frac{p}{1-p}>0$. Moreover, after some algebra we see that at $z=b$ the right hand side of this expression yields $\sG_1(b)$. Conversely, if $b \leq z$ then $f^{b,d}(z) = \sG_1(z)$. It follows that 
    \begin{eqnarray*}
        \sup_{b: b \geq 0} f^{b,d}(z) & = & 
        \sup_{b: b \geq z} \left\{ \frac{\sL_1(-d)}{\phi(d)} \phi(z) + \left(  \psi(z) +  \frac{\psi(d)}{\phi(d)}\frac{p}{1-p} \phi(z) \right) \Gamma_{1,d}(b) \right\} \\
        & = &  \frac{\sL_1(-d)}{\phi(d)} \phi(z) + \left(  \psi(z) +  \frac{\psi(d)}{\phi(d)}\frac{p}{1-p} \phi(z) \right)\sup_{b: b \geq z} \Gamma_{1,d}(b).
    \end{eqnarray*}
    Then, if $z \leq b^*(d)$, $\arg \sup_{b \geq z} f^{b,d}_+(z) = b^*(d)$. Similarly, using the fact that $\Gamma_{1,d}$ is decreasing above $b^*(d)$, if $z > b^*(d)$ we have $z \in \arg \sup_{b \geq z} f^{b,d}_+(z) = z$. In the former case we have argued directly that the choice $b^*(d)$ is optimal. In the latter case, the choice $b=z$ is optimal and $f^{b,d}(z)= \sG_1(z)$. But the choice $b=b^*(d)$ also  
    yields $f^{b^*(d),d}(z) = \sG_1(z)$; hence $b^*(d)$
    is again optimal.

    Now consider the case $d=0$. Suppose $b^{*}=b^*(0)$ is the maximiser of $\Gamma_{1,d}$. From \eqref{eq:fdef_0}, we know when $z\leq 0$, $f^{b,0}(z)=\mathcal{L}_{1}(z)$ which does not depend on $b$, thus $b^{*}$ maximises $f^{b,0}(z)$ when $z\leq0$.  
    To show $b^{*}$ maximises $f^{b,0}(z)$ when $z>0$, we can apply the same argument as in the case $d>0$. 

\end{proof}

\begin{corollary}\label{Cor:valuebddpayoffs} Fix $d \geq 0$.
 For any $z\in [0,\infty)$ we have $f^{b^*(d),d}(z)\geq \mathcal{G}_{1}(z)$. 
\end{corollary}
\begin{proof}
The inequality $f^{b^*(d),d}(z)\geq \mathcal{G}_{1}(z)$ is immediate (as an equality) for $z \geq b^*(d)$ so assume $0\leq z<b^* = b^*(d)$. 

Note that from \eqref{ValueMatching_1} and \eqref{ValueMatching_0_1} (for $b=z$), for any $d\geq 0$ we have that $\sG_1(z) = A(z,d) \phi(z) + B_+(z,d) \psi(z)$.
Then, for $0\leq z<b^*$,
\[ 
 f^{b^*(d),d}(z) - \mathcal{G}_{1}(z)  = (A(b^*(d),d) - A(z,d))\phi(z) + (B_+(b^*(d),d) - B_+(z,d)) \psi(z) 
 \geq  0     
\] 
since $b^*(d)$ maximises both $A(b,d)$ and $B_+(b,d)$ for fixed $d$.
\end{proof}

Fix $d>0$. Following \eqref{eq:fdef}, we define the candidate value function  
\begin{equation}\label{Expression_V_star}
    \begin{split}
    V_{*,d}(t,x)=
    \begin{cases}
        \frac{1}{\sqrt{1+t}}\left[A(b^*,d)\phi(\frac{x}{\sqrt{1+t}})+B_{+}(b^*,d)\psi(\frac{x}{\sqrt{1+t}})\right],
         & \ 0 \leq x< b^*\sqrt{1+t},\\
         \frac{1}{\sqrt{1+t}}\left[A(b^*,d)\phi(\frac{x}{\sqrt{1+t}})+B_{-}(b^*,d)\psi(\frac{x}{\sqrt{1+t}})\right],
         & \ -d\sqrt{1+t}<x<0,\\
        \mathbb{G}_{1}(t,x),  & \ x\geq b^*\sqrt{1+t}, \\
        \mathbb{L}_{1}(t,x),  & \  x\leq -d\sqrt{1+t},
    \end{cases}
    \end{split} 
\end{equation}
where $A(b^*,d),B_+(b^*,d)$ and $B_-(b^*,d)$ are given in \eqref{Expression_coefficients_2playersA} \eqref{Expression_coefficients_2playersB+} and \eqref{Expression_coefficients_2playersB-}, respectively and $b^* = b^*(d)$.
For $d=0$, define 
\begin{equation}\label{Expression_V_star_0}
    \begin{split}
    V_{*,0}(t,x)=
    \begin{cases}
        \frac{1}{\sqrt{1+t}}\left[A(b^*,0)\phi(\frac{x}{\sqrt{1+t}})+B_{+}(b^*,0)\psi(\frac{x}{\sqrt{1+t}})\right],
         & \ 0< x< b^*\sqrt{1+t},\\
        \mathbb{G}_{1}(t,x),  & \ x\geq b^*\sqrt{1+t}, \\
        \mathbb{L}_{1}(t,x),  & \  x\leq 0,
    \end{cases}
    \end{split} 
\end{equation}
where $A(b^*,0)$ and $B_{+}(b^*,0)$ are given in \eqref{ExpressionAB_0} and $b^*=b^*(0)$.

\begin{assumption}\label{supermartingale_assumption}
    $-\mathcal{G}_{1}(z)-z\mathcal{G}'_{1}(z)+\mathcal{G}''_{1}(z)\leq 0$ for any $z\in [b^*,\infty)$.
\end{assumption}

\begin{theorem}[Verification theorem]\label{Verification_theorem_2player}
    Suppose Player 2 uses the threshold strategy $\sigma^d :=\inf\{s\geq 0: X_{s}\leq - d \sqrt{1+s} \}$ where $d \in [0,\infty)$. Under Assumption \ref{ass:Gamma} and \ref{supermartingale_assumption}, the value function $\overline{V}_{1,\sigma^d}$ defined in \eqref{Valuefunc_P1} coincides with the value function $V_{*,d}(t,x)$ given by \eqref{Expression_V_star} (or $V_{*,0}(t,x)$ given by \eqref{Expression_V_star_0} if $d=0$). Moreover, the stopping time 
    \begin{equation*}
        \tau^{*}:=\inf\{s\geq 0: X_{s}\geq b^*\sqrt{1+s}\},
    \end{equation*}
    where $b^*=b^*(d)$ is the maximizer of $\Gamma_{1,d}$, is optimal. 
\end{theorem}

\begin{proof}
    We establish the result for the case $d>0$. The case $d=0$ follows similarly. 

    Suppose $d>0$.  Fix $t_{0}\in[0,\infty)$, we show that $V_{*,d}(t_{0},x)=\overline{V}_{1,\sigma^d}(t_{0},x)$ by first proving $V_{*,d}(t_{0},x)\geq \overline{V}_{1,\sigma^d}(t_{0},x)$, and then proving the reverse inequality.

    For $t_{0}\leq t<\infty$, let $M=(M_{t})_{t_{0}\leq t<\infty}$ be given by $M_{t}=V_{*,d}(t \wedge \sigma^d,X_{t \wedge \sigma^d})$. Using the It\^o-Tanaka formula (see Revuz and Yor~\cite[Chapter~VI, p.~234]{revuz_yor_1999}), we have 
    \begin{equation*}
       \begin{split}
           dM_{t}&= \mathbb{I}_{ \{ t < \sigma^d \} } \left.  \left[\Dot{V}_{*,d}(t,X_{t})+\frac{1}{2}V_{*,d}''(t, X_{t})\right]dt+ \mathbb{I}_{ \{ t < \sigma^d \} } V_{*,d}'(t, X_{t})dB_{t} \right. \\
           &\hspace{50mm}  +\mathbb{I}_{ \{ t < \sigma^d \} } \left[pV_{*,d}'(t,0^{+})-(1-p)V_{*,d}'(t,0^{-})\right]d\Tilde{L}_{t}^{0} . 
       \end{split}
    \end{equation*}
    Recalling \eqref{Valuefunc_PDE}-\eqref{ConditionLocalTime_original}, in the continuation region, i.e. $-d\sqrt{1+t}<X_{t}<b^*\sqrt{1+t}$, both the drift term and the local time term vanish. Moreover, when $x\leq -d\sqrt{1+t}$ the game is stopped immediately by Player 2.
    Thus, 
    \begin{equation*}
      \begin{split}
        dM_{t}= \mathbb{I}_{ \{ t < \sigma^d \} } \left[\Dot{\mathbb{G}}_{1}(t,X_{t})+\frac{1}{2}\mathbb{G}_{1}''(t,X_t)\right]I_{[b^*\sqrt{1+t},\infty)}(X_{t})dt+ \mathbb{I}_{ \{ t < \sigma^d \} }V_{*,d}'(t, X_{t})dB_{t}.
      \end{split}
  \end{equation*}
    Recall that $\mathbb{G}_{1}(t,x)=\frac{1}{\sqrt{1+t}}\mathcal{G}_{1}(\frac{x}{\sqrt{1+t}})$. Then, with $Z_t = \frac{X_t}{\sqrt{1+t}}$,
  \begin{equation*}
       dM_{t} =  \mathbb{I}_{ \{ t < \sigma^d \} }\frac{1}{2}(1+t)^{-\frac{3}{2}}\left[-\mathcal{G}_{1}\left( Z_t\right)-Z_t\mathcal{G}'_{1}\left(Z_t\right)+\mathcal{G}''_{1}\left( Z_t \right)\right]I_{[b^*,\infty)}(Z_{t})dt
       + \mathbb{I}_{ \{ t < \sigma^d \} } V_{*,d}'(t, X_{t})dB_{t}.
  \end{equation*}
  Let $A$ be the decreasing process given by $$A_t = \int_0^{t \wedge \sigma^d}\frac{1}{2}(1+t)^{-\frac{3}{2}}\left[-\mathcal{G}_{1}\left( Z_s\right)-Z_s\mathcal{G}'_{1}\left(Z_s\right)+\mathcal{G}''_{1}\left( Z_s \right)\right]I_{[b^*,\infty)}(Z_{s})ds$$ and
let $N$ be the local martingale given by $N_t = \int_0^{t \wedge \sigma^d}  V_{*,d}'(t, X_{s})dB_{s}$. Then $M_t = A_t + N_t \leq N_t$.
Moreover, since $V_{*,d}$ is increasing in $x$, $M_t \geq V_{*,d}(t \wedge \sigma^d, X_{\sigma^d}) \geq  \frac{1}{\sqrt{1+ t \wedge \sigma^d}} \sL_1(-d) \geq \sL_1(-d) \wedge 0$ and $M$ is bounded below. Then $N$ is also bounded below, and hence a supermartingale. {\em A fortiori}, $M$ is a supermartingale.

    By the Optional Stopping theorem for a supermartingale, it follows that for $\tau \in \mathcal{T}_{t_0}$, 
    \begin{equation*}
    \begin{split}
      V_{*,d}(t_{0},x)&\geq \mathbb{E}[V^{*,d}(\tau \wedge \sigma^d_{t_{0}},X_{\tau \wedge \sigma^d_{t_{0}}})]\\
      &= \mathbb{E}[V_{*,d}(\tau,X_{\tau})\mathbb{I}_{\{\tau\leq\sigma^d_{t_{0}}\}}+V_{*,d}(\sigma^d_{t_{0}},X_{\sigma^d_{t_{0}}})\mathbb{I}_{\{\sigma^d_{t_{0}}<\tau\}}], \\
      &\geq \mathbb{E}[\mathbb{G}_{1}(\tau,X_{\tau})\mathbb{I}_{\{\tau\leq\sigma^d_{t_{0}}\}}+\mathbb{L}_{1}(\sigma^d_{t_{0}},X_{\sigma^d_{t_{0}}})\mathbb{I}_{\{\sigma^d_{t_{0}}<\tau\}}].
    \end{split}
    \end{equation*}
    Therefore,  
    $$V_{*,d}(t_{0},x)\geq \sup_{\tau \in \sT_{t_{0}}}\mathbb{E}_{t_{0},x}[\mathbb{G}_{1}(\tau,X_{\tau})\mathbb{I}_{\{\tau<\sigma^d_{t_{0}}\}}+\mathbb{L}_{1}(\sigma^d_{t_{0}},X_{\sigma^d_{t_{0}}})\mathbb{I}_{\{\sigma^d_{t_{0}}<\tau\}}]=\overline{V}_{1,\sigma^d}(t_{0},x). $$

    Now we aim to show the reverse inequality i.e. $\overline{V}_{1,\sigma^d}(t_{0},x)\geq V_{*,d}(t_{0},x)$. When $x\geq b^*\sqrt{1+t}$, from Corollary~\ref{Cor:valuebddpayoffs} we have $\overline{V}_{1,\sigma^d}(t_{0},x)\geq \mathbb{G}_{1}(t_{0},x)=V_{*,d}(t_{0},x)$. Moreover, since Player 2 stops whenever $X_t \leq -d\sqrt{1+t}$, we have $\overline{V}_{1,\sigma^d}(t_{0},x)= \mathbb{L}_{1}(t_{0},x)=V_{*,d}(t_{0},x)$ when $x\leq -d\sqrt{1+t}$. It remains to consider $x$ in the continuation region, i.e. $-d\sqrt{1+t}< x< b^*\sqrt{1+t}$. 
    
Let $\tau^*$ be as in the statement of the theorem. Let $M^* = (M_{t \wedge \tau^*})_{0 \leq t < \infty}$. Then $dM^*_t = I_{ \{ t < \sigma^d \wedge \tau^* \} }V_{*,d}'(t, X_{t})dB_{t}$. On the continuation region there is a formula for $V_{*,d}$ and $V'_{*,d}$ is bounded. Then $M^*$ is a martingale. 
By the Optional Stopping theorem again, we have
     \begin{equation*}
    \begin{split}
      V_{*,d}(t_{0},x)&= \mathbb{E}_{t_{0},x}[V_{*,d}(\tau^{*}\wedge \sigma^d_{t_{0}},X_{\tau^{*}\wedge \sigma^d_{t_{0}}})],\\
       &= \mathbb{E}_{t_{0},x}[V_{*,d}(\tau^{*},X_{\tau^{*}})\mathbb{I}_{\{\tau^{*}\leq\sigma^d_{t_{0}}\}}+V_{*,d}(\sigma^d_{t_{0}},X_{\sigma^d_{t_{0}}})\mathbb{I}_{\{\sigma^d_{t_{0}}<\tau^{*}\}}],\\
       &=\mathbb{E}_{t_{0},x}[\mathbb{G}_{1}(\tau^{*},X_{\tau^{*}})\mathbb{I}_{\{\tau^{*}\leq\sigma^d_{t_{0}}\}}+\mathbb{L}_{1}(\sigma^d_{t_{0}},X_{\sigma^d_{t_{0}}})\mathbb{I}_{\{\sigma^d_{t_{0}}<\tau^{*}\}}].
    \end{split}
    \end{equation*}
    where the final equality comes from the definition of $V^{*,d}$ on the stopping boundaries. Hence, we have 
    \begin{equation*}
        \begin{split}
         \overline{V}_{1,\sigma^d}(t_{0},x)&= \sup_{\tau \in \sT_{t_{0}}}\mathbb{E}_{t_{0},x}[\mathbb{G}_{1}(\tau,X_{\tau})\mathbb{I}_{\{\tau\leq\sigma^d_{t_{0}}\}}+\mathbb{L}_{1}(\sigma^d_{t_{0}},X_{\sigma^d_{t_{0}}})\mathbb{I}_{\{\sigma^d_{t_{0}}<\tau\}}] \\
         &\geq \mathbb{E}_{t_{0},x}[\mathbb{G}_{1}(\tau^{*},X_{\tau^{*}})\mathbb{I}_{\{\tau^{*}\leq\sigma^d_{t_{0}}\}}+\mathbb{L}_{1}(\sigma^d_{t_{0}},X_{\sigma^d_{t_{0}}})\mathbb{I}_{\{\sigma^d_{t_{0}}<\tau^{*}\}}]=V_{*,d}(t_{0},x).
        \end{split}
    \end{equation*}

\end{proof} 

\subsection{Existence of a Subgame-Perfect Nash Equilibrium in the Two-Player Game}

In this section, we study the two-player optimal stopping game introduced in Section 2.1, our goal being to find sufficient conditions for the existence of a subgame-perfect Nash equilibrium (SPNE) in this game setup. (We study uniqueness in the next section.)

In the last section, given Player 2 uses a square-root threshold strategy $\sigma^d$, we found that under Assumption \ref{ass:Gamma} and \ref{supermartingale_assumption}, the optimal stopping strategy for Player 1 is to also consider a square-root threshold strategy given by $\tau^b=\inf\{t\geq 0: X_t\geq b\sqrt{1+t}\}$ where $b$ is the global maximiser of the function $\Gamma_{1,d}$ given in \eqref{Gamma_expression_}. $\tau^{b^*}$ is optimal in the class of all stopping rules.

Given that Player 1 considers a square-root threshold strategy $\tau^b$, we expect that the optimal stopping strategy $\sigma^{d}$ for Player 2 is also of square-root threshold form. Define $\Gamma_{2,b}=\Gamma_{2,b, p, \alpha, \theta} : [0,\infty) \to \RR$ by 
\begin{equation}\label{Gamma_Expression_2}
   \Gamma_{2,b}(d)=p \left[\frac{\mathcal{G}_{2}(-d)\phi(b)-\phi(d)\mathcal{L}_{2}(b)}{(1-p)\psi(b)\phi(d)+p\phi(b)\psi(d)}\right].
\end{equation}
We make analogous assumptions for Player 2:
\begin{assumption}
\label{ass:Gamma_2}
 $\Gamma_2$ is maximised at $d^*=d^*(b) = d^*(b,p,\alpha,\theta) \in [0,\infty)$ and $\Gamma_2$ is non-increasing on $[d^*,\infty)$.
\end{assumption}

\begin{assumption}\label{supermartingale_assumption_2}
    For any $z\in (-\infty,-d^*]$, it holds that $-\mathcal{G}_{2}(z)-z\mathcal{G}'_{2}(z)+\mathcal{G}''_{2}(z)\leq 0$.
\end{assumption}
Then, repeating the same analysis as in the last section, we know that under Assumptions \ref{ass:Gamma_2} and \ref{supermartingale_assumption_2}, the optimal stopping strategy for Player 2 is given by $\sigma^{d^*}=\inf \{t\geq 0: X_t\leq -d^* \sqrt{1+t}\}$ where $d^*\in[0,\infty)$ is the maximiser of the function $\Gamma_{2,b}$.

Our goal is to find sufficient conditions such that there exists a pair $(b^*,d^*) \in (0,\infty)^2$ such that $\Gamma_{1,d^*}$ attains its maximum at $b^*=b^*(d^*)$ and $\Gamma_{2,b^*}$ attains its maximum at $d^*=d^*(b^*)$. The pair $(b^*,d^*)$ will then characterise a candidate SPNE.

Define $h_1(z)=\mathcal{G}_{1}'(z)-z\mathcal{G}_{1}(z)$ and $h_2(z)=\mathcal{G}_{2}'(-z)-z\mathcal{G}_{2}(-z)$.
When we say $h :[0,\infty) \to \RR$ is unimodal we mean that $h$ is $C^1$ on $[0,\infty)$ and there exists a unique $z^* \in [0,\infty]$ such that $h'>0$ on $[0,z^*)$ and $h'<0$ on $(z^*,\infty)$. 

\begin{assumption}\label{Assumption_NE_general}
    \begin{enumerate}
     \item[(i)] There exist solutions $\hat{b}\in(0,\infty)$ to $h_1(\cdot)=0$ and $\hat{d} \in (0,\infty)$ to $h_{2}(\cdot)=0$. 
     \item[(ii)] For every $d\in[0,\hat{d}]$, we have that 
     \begin{enumerate}
         \item $\Gamma_{1,d}(\hat{b})>0$,
         \item $\Gamma_{1,d}$ is unimodal, and attains its maximum at $b^*=b^*(d) = b^*(d,p,\alpha,\theta) \in (0,\infty)$. 
     \end{enumerate}    
     \item[(iii)] For every $b\in[0,\hat{b}]$, we have that 
     \begin{enumerate}
         \item $\Gamma_{2,b}(\hat{d})>0$.
         \item $\Gamma_{2,b}$ is unimodal, and attains its maximum at $d^*=d^*(b) = d^*(b,p,\alpha,\theta) \in (0,\infty)$. 
    \end{enumerate}
    \end{enumerate}
\end{assumption}
Note that Assumption \ref{Assumption_NE_general}(ii)(b) implies Assumption \ref{ass:Gamma} and \ref{Assumption_NE_general}(iii)(b) implies Assumption~\ref{ass:Gamma_2}.

\begin{theorem}\label{Main_Theorem_NE}
    Under Assumptions \ref{supermartingale_assumption}, \ref{supermartingale_assumption_2}, and \ref{Assumption_NE_general}, there exists a subgame-perfect Nash equilibrium for the two-player optimal stopping game introduced in Section 2.1.  
\end{theorem}

\begin{proof}
    For the purpose of this proof we write $\Gamma_{1,d}$ and $\Gamma_{2,b}$ given in \eqref{Gamma_expression_} and \eqref{Gamma_Expression_2} as bivariate functions:
    \begin{equation*}
        \begin{split}
             \Gamma_{1}(b,d) &=(1-p)\left[\frac{\mathcal{G}_{1}(b)\phi(d)-\phi(b)\mathcal{L}_{1}(-d)}{(1-p)\psi(b)\phi(d)+p\phi(b)\psi(d)}\right],\\
             \Gamma_2(b,d)&=p \left[\frac{\mathcal{G}_{2}(-d)\phi(b)-\phi(d)\mathcal{L}_{2}(b)}{(1-p)\psi(b)\phi(d)+p\phi(b)\psi(d)}\right].
        \end{split}
    \end{equation*}
It follows that 
\begin{equation*}
    \begin{split}
        \frac{\partial \Gamma_{1}(b,d)}{\partial b}&= \frac{(1-p)\phi(d)}{(1-p)\psi(b)\phi(d)+p\phi(b)\psi(d)}\left( h_1(b)-\Gamma_{1}(b,d)\right),\\
        \frac{\partial \Gamma_{2}(b,d)}{\partial d}&= \frac{p\phi(b)}{(1-p)\psi(b)\phi(d)+p\phi(b)\psi(d)}\left(h_2(d)-\Gamma_{2}(b,d)\right). 
    \end{split}
\end{equation*}

For $i=1,2$, define $T_i :[0,\hat{b}]\times[0,\hat{d}] \to \RR$ by $T_i(b,d) =h_i(b) - \Gamma_i(b,d)$. Let $T: [0,\hat{b}]\times[0,\hat{d}] \to \RR^2$ be given by $T(b,d)=(T_1(b,d),T_2(b,d))$. Then $T$ is a continuous function since $h_{1}, h_{2}, \Gamma_1$ and $\Gamma_2$ are all continuous functions. 

Under Assumption \ref{Assumption_NE_general}(i) there exists $\hat{d}$ such that $h_2(\hat{d})=0$, and then by Assumption \ref{Assumption_NE_general}(ii)(b), for $d \in [0,\hat{d}]$ there exists $b^*(d)$ such that $b^*$ is the global maximiser of $\Gamma_{1}(\cdot,d)$. It follows that $h_1(b) > \Gamma_1(b,d)$ on $[0,b^*(d))$ and hence 
$T_1(0,d)>0$. Conversely, since $h_1(\hat{b})=0<\Gamma_{1,d}(\hat{b})$ we have $T_1(\hat{b},d)<0$.
Similarly, for all $b \in [0,\hat{b}]$, $T_2(b,0)>0>T_2(b,\hat{d})$.
It follows from Miranda's theorem \cite{Miranda1940} that there exists $(b^{*}, d^{*})\in D$, such that $(T_{1}(b^{*}, d^{*}),T_{2}(b^{*}, d^{*}))=(0,0)$.

It remains to show that $(\tau^*,\sigma^*)$ is a SPNE where $\tau^*= \tau^{b^*}$ and $\sigma^* = \sigma^{d^*}$.
Assumption~\ref{Assumption_NE_general} implies that
Assumption \ref{ass:Gamma} 
holds. Then, by Theorem~\ref{Verification_theorem_2player}, $\tau^{b^*}$ is optimal for Player 1 (from any starting point $(t_0,X_{t_0}=x) \in [0,\infty) \times \RR$).
Similarly, Assumption \ref{ass:Gamma_2} holds  and by the direct analogue of Theorem~\ref{Verification_theorem_2player} but for Player 2, $\sigma^{d^*}$ is optimal for Player 2. Indeed, $(\tau^*,\sigma^*)$ forms an SPNE of threshold type. 

\end{proof}

\begin{remark}
\label{rem:nonredundancy}
    It is clear that a threshold-type Nash equilibrium is automatically an SPNE. Thus, to avoid redundancy, in the following sections we will use the term "threshold-type Nash equilibrium" instead of "threshold-type SPNE".   
\end{remark}

We now give some simpler sufficient conditions under which Assumptions~\ref{supermartingale_assumption}, \ref{supermartingale_assumption_2} and \ref{Assumption_NE_general} hold.

\begin{prop}\label{Suff_conditions_ass5}
    Suppose that for all $z\in(0,\infty)$ we have $\mathcal{G}'_1(z)>0$ and $\mathcal{L}_2(z)<0$ and for all $z\in(-\infty,0)$ we have $\mathcal{G}'_2(z)<0$ and $\mathcal{L}_1(z)<0$. Suppose that there exist solutions $\hat{b}\in(0,\infty)$ to $h_1(\cdot)=0$ and $\hat{d} \in (0,\infty)$ to $h_{2}(\cdot)=0$. 
    \begin{enumerate}
     
\item[(i)]   Suppose that $h_1$ is unimodal on $[0,\infty)$ with a unique maximiser $z_{h_1}^*\in (0, \hat{b})$, and suppose that
    $h_1(0)>\Gamma_{1,d}(0)$ for all $d\in[0,\hat{d}]$.  
    Then, Assumption \ref{supermartingale_assumption} and \ref{Assumption_NE_general}(ii) hold.

\item[(ii)]    Suppose that $h_2$ is unimodal on $[0,\infty)$ with a unique maximiser $z_{h_2}^*\in (0, \hat{d})$, and suppose that
    $h_2(0)>\Gamma_{2,b}(0)$ for all $b\in[0,\hat{b}]$.  
    Then, Assumption \ref{supermartingale_assumption_2} and \ref{Assumption_NE_general}(iii) hold.
    \end{enumerate}
\end{prop}

\begin{proof}
We prove $(i)$. $(ii)$ follows similarly.

Recall that $h_1(z)=\mathcal{G}'_1(z)-z\mathcal{G}_1(z)$, and $\hat{b}$ is the solution to $h( \cdot )=0$.

    Since $\mathcal{G}'_1(\hat{b})>0$ and $\hat{b} \mathcal{G}_1(\hat{b})= \mathcal{G}'_1(\hat{b})$, we immediately have that $\mathcal{G}_1(\hat{b})>0$. Then, since by assumption 
    $\mathcal{L}_1(-d)<0$ for any $d \geq 0$, we have
    from the definition of $\Gamma_{1,d}$ in \eqref{Gamma_expression_} that $\Gamma_{1,d}(\hat{b})>0$ for any $d \geq 0$.

    Let $\Upsilon_1:(0,\infty) \to \RR$ be the continuous function given by $\Upsilon_1(z)=\Gamma_{1,d}(z)-h_1(z)$. Note that $\Upsilon_1(\hat{b})>0$. 
    Since $h_1(0)>\Gamma_{1,d}(0)$ we have that $\Upsilon_1(0)<0$.  Hence, by the Intermediate Value Theorem there exists at least one root such that $\Upsilon_1(z)=0$. We let $z^*$ be the first root of $\Upsilon_1$, we argue that it is the only root. Note that we must have $\Upsilon_1'(z^*) \geq 0$.
    Note that if $\Upsilon_1(z)=0$ then $\Gamma'_{1,d}(z)=0$, and $\Upsilon'(z) = - h'(z)$.
    
    Let $z^*_{h_1}$ be the unique maximum of $h_1$. If $z^* < z^*_{h_1}$
    then $h'(z^*)>0$ and $\Upsilon_1'(z^*) < 0$, a contradiction. So $z^* \geq z^*_{h_1}$. Now suppose $z > z^* \geq z^*_{h_1}$. If $z$ is a root of $\Upsilon_1$, then $\Upsilon_1'(z) = - h'(z) > 0$ and hence $z$ is an upcrossing. But, if every root of $\Upsilon_1$ is an upcrossing, then since $\Upsilon_1$ is continuous there can only be one root. It also follows that $z^*$ is an upcrossing, and that $z^* \in (z^*_{h_1},\hat{b})$. Let $b^* = b^*(d) = z^*$. 

    Recall that $$\Gamma_{1,d}'(b) = C_1(p,b,d)(h_1(b) - \Gamma_{1,d}(b) )= - C_1(p,b,d)\Upsilon_1(b)$$ where $C_1(p,b,d) = \frac{(1-p)\phi(d)}{(1-p)\psi(b)\phi(d) + p \psi(d)\phi(b)}$. Then since $\Upsilon_1 < 0$ on $(0,b^*)$ and $\Upsilon_1>0$ on $(b^*,\infty)$ we have that $\Gamma_{1,d}$ 
    is strictly increasing on $(0,b^*)$ and strictly decreasing on $(b^*, \infty)$. Then Assumption \ref{Assumption_NE_general}(ii) holds.

    Since $h_1$ is strictly decreasing on $(z^*_{h_1},\infty)$ we have that $h_1' \leq 0$ on $[b^*,\infty) \subseteq (z^*_{h_1},\infty)$ so that $\mathcal{G}_1''(z) - z\mathcal{G}_1'(z) -  \mathcal{G}_1(z) \leq 0$ on this same set. Then Assumption \ref{supermartingale_assumption} holds.

\end{proof}

\subsection{A linear example}
\label{eg:linearexistence}
Previously, we have established sufficient conditions for the existence of a SPNE for general payoff functions. Now we specialize to the case of linear payoff functions. More specifically, we consider the case where
\begin{eqnarray*} 
\mathcal{G}_{1},\mathcal{L}_{2}:[0,\infty)\mapsto\mathbb{R} & & \mathcal{G}_{1}(z)=z-\alpha_1 \hspace{10mm} \mathcal{L}_{2}(z)=-\theta_2 z, \\
\mathcal{G}_{2},\mathcal{L}_{1}:(-\infty,0]\mapsto\mathbb{R} & & \mathcal{G}_{2}(z)=-z-\alpha_2 \hspace{10mm} \mathcal{L}_{1}(z)=\theta_1 z.
\end{eqnarray*}

First, we aim to find sufficient conditions for the existence of an SPNE. Later, we will study the uniqueness of the SPNE.

\begin{theorem}\label{AsymmetricTheorem-2player} Suppose $\sG_i$ and $\sL_i$ are as above for $i=1,2$.
Suppose $0<p<1$.

Suppose that $\theta_{1}\leq \frac{p}{1-p}$ and $\theta_{2}\leq\frac{1-p}{p}$. Then there exists a threshold-type Nash equilibrium.
\end{theorem}

\begin{proof}
    This is a direct application of Theorem \ref{Main_Theorem_NE}, where we use Proposition~\ref{Suff_conditions_ass5} to verify that Assumptions~\ref{supermartingale_assumption}, \ref{supermartingale_assumption_2} and \ref{Assumption_NE_general} all hold. 
    
    From the definitions of $\mathcal{G}_1,\mathcal{G}_2,\mathcal{L}_1$ and $\mathcal{L}_2$, it is follows that 
    $h_1(z)=1-z(z-\alpha_1)$ and $h_2(z)=1-z(z-\alpha_2)$. Let $\hat{b}$ and $\hat{d}$ denote the (positive) roots of $h_1$ and $h_2$, respectively. 
    It suffices to show that the conditions $\theta_1 \leq \frac{p}{1-p}$ and $\theta_2 \leq \frac{1-p}{p}$ are sufficient conditions to guarantee that $\Gamma_{1,d}(0)<h_{1}(0)$ for any $d\in[0, \hat{d}]$, and $\Gamma_{2,b}(0)<h_{2}(0)$ for any $b\in[0,\hat{d}]$. Then by Proposition \ref{Suff_conditions_ass5} and Theorem \ref{Main_Theorem_NE}, it follows that there exists a threshold-type Nash equilibrium. 

    Suppose $\theta_{1}\leq \frac{p}{1-p}$. We then have 
\[
        \Gamma_{1,d}(0) < h_1(0) \Longleftrightarrow  (1-p)\frac{-\alpha_{1}\phi(d)+\theta_{1}d}{p\psi(d)}< 1
            \Longleftrightarrow R(d)>0
\]   
where
\[ R(d)=p\int_{0}^{d}e^{-\frac{u^{2}}{2}}du+(1-p)\alpha_{1}-(1-p)\theta_{1}de^{-\frac{d^{2}}{2}}. \] 
For all $d \in (0,\infty)$ we have  $R'(d)=[p-(1-p)\theta_{1}]e^{-\frac{d^{2}}{2}}+ (1-p)\theta_{1}d^{2}e^{-\frac{d^{2}}{2}}>0$, since we are assuming that $\theta_{1}\leq \frac{p}{1-p}$. Hence, $R$ is increasing on $[0,\infty)$ and since $R(0)=(1-p)\alpha_{1}>0$ we conclude that $R(d)> 0$ for all $d\in[0,\infty)$. 
With similar arguments, we can show that $\theta_{2}\leq \frac{1-p}{p}$ is a sufficient condition for $\Gamma_{2,b}(0)<h_2(0)$.

\end{proof}

\begin{figure}[h]
    \centering
    \begin{subfigure}[t]{0.48\textwidth}
        \vspace{0pt}
        \includegraphics[width=\linewidth]{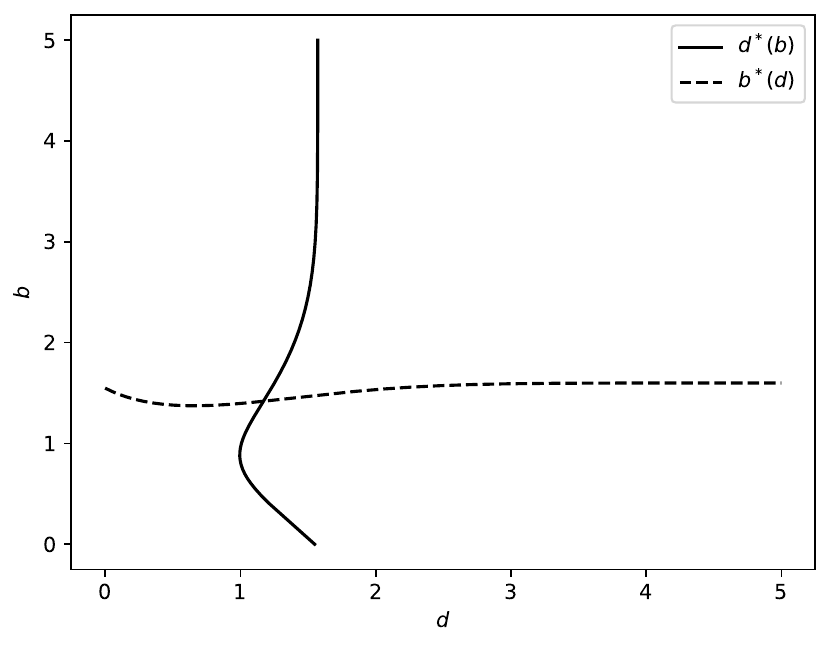}
        \caption{$\theta_1=\theta_2=2$ and $\alpha_1=\alpha_2=1$. The sufficient conditions given in Theorem \ref{AsymmetricTheorem-2player} are satisfied, there is an unique non-zero threshold-type Nash equilibrium (the unique intersection between the two curves, where $b=b^*(1.19)\approx 1.46$ and $d=d^*(1.46)\approx 1.19$).}
        \label{fig:Asymmetric_NE_small_theta}
    \end{subfigure}
    \hfill
    \begin{subfigure}[t]{0.48\textwidth}
        \vspace{0pt}
        \includegraphics[width=\linewidth]{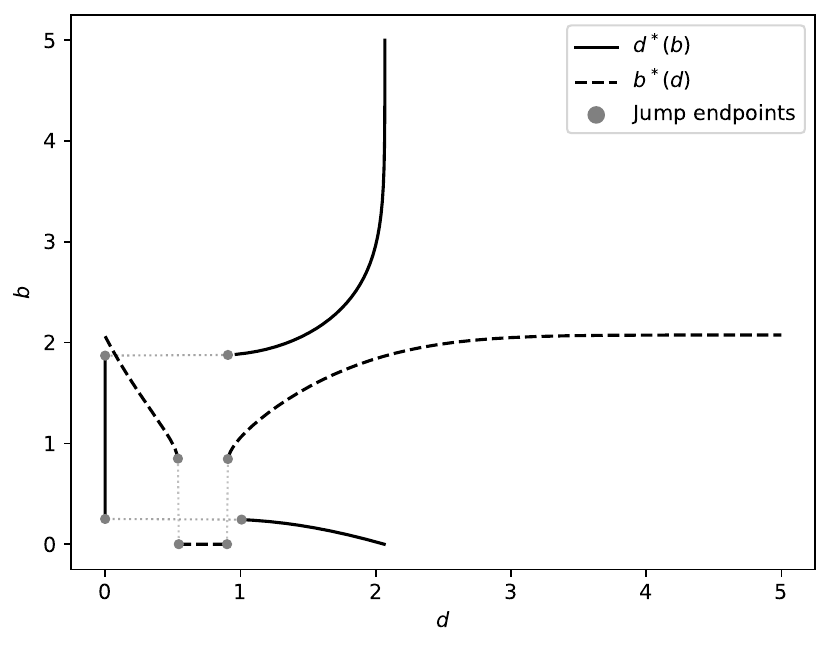}
        \caption{$\theta_1=\theta_2=8$ and $\alpha_1=\alpha_2=1.6$. The sufficient conditions given in Theorem \ref{AsymmetricTheorem-2player} fail, and there is no threshold-type Nash equilibrium.}        \label{fig:Asymmetric_NE_large_theta}
    \end{subfigure}
    \caption{The existence of threshold-type Nash equilibria in the asymmetric case where $p=0.7$}
    \label{fig:twoplots}
\end{figure}

\begin{remark}
\label{rem:counterexampleHamedene}
In some cases we find that there is no threshold-type Nash equilibrium. Nonetheless, Hamed\`ene and Zhang~\cite{HamadeneZhang2010} (and Laraki and Solan~\cite{LarakiSolan2013}) prove existence of a Nash equilibrium. The difference arises from the different notion of Nash equilibia used in those two papers. Unlike Hamed\`ene and Zhang~\cite{HamadeneZhang2010} and Laraki and Solan~\cite{LarakiSolan2013} we look for subgame-perfect Nash equilibria. Effectively, Hamed\`ene and Zhang~\cite{HamadeneZhang2010} and Laraki and Solan~\cite{LarakiSolan2013} consider the problem started from $(t=0,X_0=0)$ only.

If, for example, we are in the setting of Figure~\ref{fig:twoplots}(b) then taking $\tau=\tau_b$ with $b= 1.5$ we find that the optimal response is $\sigma = \sigma_d$ where $d= d(1.5)=0$. For the game started at $(t=0,X_0=0)$, the strategies $(\tau=\tau_{1.5},\sigma=0)$ is a Nash equilibrium: if $\tau=\tau_{1.5}$ then $\sigma=0$ is the optimal response, and if $\sigma = 0$ then any $\tau$ with $\tau>0$ is optimal. However, this optimality does not carry over to general starting points: $(\tau=\tau_{1.5},\sigma=0)$ is not a subgame-perfect Nash equilibrium, and the the optimal response to $\sigma=0$ is to take $\tau = \tau_{b^*(0)}$ where $b^*(0) \sim 2.1$ (and then the optimal response (of threshold-type) to $\tau_{b^*(0)}$ is $\sigma_{d^*(b^*(0))} \neq 0$). Hence, the results of this paper are not inconsistent with the results of \cite{HamadeneZhang2010} and \cite{LarakiSolan2013} since we show there is no SPNE of threshold-type which is simultaneously optimal for all starting points, where as Hamad\`ene and Zhang~\cite{HamadeneZhang2010} and Laraki and Solan \cite{LarakiSolan2013}
consider a problem with a fixed starting point.
\end{remark}

\subsection{Uniqueness of the Subgame-Perfect Nash Equilibrium in Symmetric Games}
So far we have established sufficient conditions for the existence of an SPNE. Now we aim to further study the properties of the game's equilibrium structure, and especially the uniqueness of the SPNE. Uniqueness is more delicate. For this reason, we focus on the case of linear payoffs as in Section~\ref{eg:linearexistence}, and the symmetric case (where $p=0.5$, and both players face identical payoffs, i.e. $\sG_1(z)=\sG_2(-z)$ and $\sL_1(z)=\sL_2(-z)$.)
We write 
$\alpha:=\alpha_1=\alpha_2$ and $\theta:=\theta_1=\theta_2$. Note that since $p=0.5$ the skew Brownian  motion is actually a standard Brownian motion. 

Since the setup is completely symmetric, it is reasonable to look for SPNE in which the strategies for the two players are the same. As in Remark~\ref{rem:nonredundancy}, the concept introduced in the next definition might more accurately be described as a symmetric SPNE of threshold-type, but we use a simpler terminology. 

\begin{definition}[Symmetric threshold-type Nash equilibrium of square-root type]
    We say that $(\tau^{*},\sigma^{*})\in\mathcal{T}^{\mathbb{SQ}}\times \mathcal{T}^{\mathbb{SQ}}$ is a symmetric Nash equilibrium within the class of square-root threshold strategies if 
    \begin{align*}
        R_{1}(\tau^{*}, \sigma^{*})&=\sup_{\tau\in\mathcal{T}^{\mathbb{SQ}}}R_{1}(\tau,\sigma^{*})\\
        R_{2}(\tau^{*}, \sigma^{*})&=\sup_{\sigma\in\mathcal{T}^{\mathbb{SQ}}}R_{2}(\tau^{*},\sigma),
    \end{align*}
    where $\tau^{*}:=\inf\{t\geq 0: X_{t}\geq b^*\sqrt{1+t}\}$, $\sigma^{*}:=\inf\{t\geq 0: X_{t}\leq -d^*\sqrt{1+t}\}$, and the stopping boundaries satisfy $b^*=d^*\in [0,\infty)$.
\end{definition}

In the previous section we established that if Player 2 uses the square-root threshold strategy $\sigma^{d}$, where $\sigma^{d}=\inf\{u: X_{u}\leq -d\sqrt{1+u}\}$ then the optimal response for Player 1 is to also consider a square-root threshold $\tau^{b}$ where $\tau^{b}=\inf\{u:X_{u}\geq b\sqrt{1+u}\}$ and $b$ is the maximizer of $\Gamma_{1,d}$ given in \eqref{Gamma_expression_}. Thus, to establish the existence of a symmetric threshold-type Nash equilibrium, it is sufficient to show that when $p=\frac{1}{2}$, $\alpha_1=\alpha_2$ and $\theta_1=\theta_2$ there exists $d\in[0,\infty)$, such that 
the maximizer $b^*(d)$ of the function $\Gamma_{d}$ satisfies $b^{*}(d)=d$. 
Here, since the problems facing the two players are the same, we drop the subscript relating to the player and use the explicit forms of $\mathcal{G}_\cdot$ and $\mathcal{L}_\cdot$. Then we can re-write \eqref{Gamma_expression_} as 
\begin{equation}\label{ExpressionGammap0.5}
  \Gamma_{d}(b)=\frac{(b-\alpha)\phi(d)+\phi(b)\theta d}{\phi(b)\psi(d)+\psi(b)\phi(d)}.
\end{equation}
We also find
\begin{equation}\label{ExpressionGamma_derivative_p0.5}
  \Gamma_{d}'(b)= \frac{\phi(d)}{\phi(b)\psi(d)+\psi(b)\phi(d)}[1-b(b-\alpha)-\Gamma_{d}(b)].
\end{equation}

Define $h_\alpha:[0,\infty) \to \RR$ by $h_{\alpha}(z)=1-z(z-\alpha)$. Note that $h_\alpha$ has a maximum at $\hat{z}(\alpha)= \frac{\alpha}{2}$ and a unique root $\Bar{z}= \bar{z}(\alpha) = \frac{\alpha}{2} + \sqrt{1 + \frac{\alpha^2}{4}}$ such that $\bar{z}>\alpha$. Define $\Theta^*,\Theta_*:(0,\infty)\times(0,\infty) \to (0,\infty)$  by
\begin{eqnarray*} 
\Theta^*(\alpha,d) & = &\frac{1}{d \phi(\frac{\alpha}{2})} \left[ \frac{\alpha}{2} \phi(d) + \left( 1 + \frac{\alpha^2}{4} \right) \psi \left(\frac{\alpha}{2} \right) \phi(d) + \phi \left( \frac{\alpha}{2} \right) \psi(d) \right]. \\ 
\Theta_*(\alpha,d) & = & \frac{\psi(d) + \alpha \phi(d)}{d}
\end{eqnarray*}
and extend the definitions (continuously) to $(0,\infty) \times [0,\infty)$ by setting $\Theta_*(\alpha,0)=\Theta^*(\alpha,0)=\infty$.
It can be shown that $2x \phi(x) < x + (1+x^2)\psi(x)$ for $x>0$ and hence that $\Theta_*(\alpha,d) < \Theta^*(\alpha,d)$ for $\alpha,d>0$.

\begin{prop}\label{p0.5_maximiser_prop}

    Fix $d>0$.
    Recall the definition of $\Gamma_d$ in \eqref{ExpressionGammap0.5} and define $\Upsilon = \Upsilon_d = \Upsilon_{\alpha,\theta,d}:[0,\infty) \mapsto \mathbb{R}$ by $\Upsilon(z) = \Gamma_d(z) - h_\alpha(z)$.
 Then $\tilde{z}$ is a turning point of $\Gamma_d$ if and only if $\tilde{z}$ is a root of $\Upsilon$.
  
The function $\Gamma_{d}$ admits a maximizer, denoted $z^{*}=z^{*}(\alpha,\theta,d)$. 
Furthermore, there exists $\hat{\Theta}:(0,\infty) \times (0,\infty)$ with $\Theta_*(\alpha,d) \leq \hat{\Theta}(\alpha,d)\leq \Theta^*(\alpha,d)$ such that
\begin{enumerate}
\item If $\theta < \Theta_*(\alpha,d)$ then
$\Upsilon$ has a unique root $z_1 \in (\frac{\alpha}{2}, \bar{z}(\alpha))$, and that root is the maximiser of $\Gamma_d$.

\item If $\theta = \Theta_*(\alpha,d)$ then 
$\Upsilon$ has two roots $z_1=0$ and $z_2 \in (\frac{\alpha}{2}, \bar{z}(\alpha))$, and $z_2$ is the maximiser of $\Gamma_d$.

\item If $\Theta_*(\alpha,d) < \theta < \hat{\Theta}(\alpha,d)$ then $\Upsilon$ has two roots $z_1 \in (0, \frac{\alpha}{2})$ and $z_2 \in (\frac{\alpha}{2}, \bar{z}(\alpha))$ and  $z_2$ is the maximiser of $\Gamma_d$.

\item If $\theta= \hat{\Theta}(\alpha,d)$ then $\Upsilon$ has two roots $z_{1}\in (0,\frac{\alpha}{2})$ and $z_2 \in (\frac{\alpha}{2}, \bar{z}(\alpha))$. Both $0$ and $z_{2}$ are maximisers of $\Gamma_d$.

\item If $\hat{\Theta}(\alpha,d) < \theta < \Theta^*(\alpha,d)$ then $\Upsilon$ has two roots $z_{1}\in (0,\frac{\alpha}{2})$ and $z_2 \in (\frac{\alpha}{2}, \bar{z}(\alpha))$. 0 is the maximiser of $\Gamma_d$.

\item If $\theta = \Theta^*(\alpha,d)$ then $\frac{\alpha}{2}$ is the unique root of $\Upsilon$ and $0$ is the maximiser of $\Gamma_d$.

\item If $\theta > \Theta^*(\alpha,d)$ then $\Upsilon$  has no roots and $0$ is the maximiser of $\Gamma_d$. 
\end{enumerate}     
\end{prop}

\begin{proof}
    Recall \eqref{ExpressionGammap0.5} and \eqref{ExpressionGamma_derivative_p0.5} and that $1-b(b-\alpha)-\Gamma_{d}(z)=h_{\alpha}(z)-\Gamma_{d}(z)=-\Upsilon(z)$.
 Then, for any $\Tilde{z}$ which satisfies
\begin{equation}\label{TurningPointCondition}
    h_{\alpha}(\Tilde{z})=\Gamma_{d}(\Tilde{z})
\end{equation}
we know that $\Tilde{z}$ is a turning point of $\Gamma_{d}$ (and vice versa).
See Figure~\ref{Gamma_func_d}.
Further, $\Gamma_{d}$ is strictly increasing if $\Gamma_d<h_\alpha$ and $\Gamma_{d}$ is strictly decreasing if $\Gamma_d>h_\alpha$ and $\Gamma_d$ is positive on $[\alpha,\infty)$.
It follows that $\lim_{z \rightarrow \infty} \Upsilon(z) =  \infty$.

If $\Gamma_d(0)<h(0)=1$ then $\Gamma_d$ is increasing until the first root $z_1$ of $\Upsilon$. By the Intermediate Value Theorem such a root must exist since $\Upsilon(0)<0$. Note that $0 \leq \Upsilon'(z_1) = -h_\alpha'(z_1)$ so that we cannot have $z_1 < \frac{\alpha}{2}$.
Further, if $z_1 \leq \alpha$ then $h(z_1)>0$ and if $z_1>\alpha$ then $\Gamma_d(z_1)>0$ so that in all cases $h(z_1)=\Gamma_d(z_1)>0$.
Finally, if $z_2>z_1> \frac{\alpha}{2}$ is a second root of $\Upsilon$ then $\Upsilon>0$ on $(z_1,z_2)$ but $\Upsilon'(z_2) = -h_\alpha'(z_2)>0$ a contradiction. Hence, if 
$\Gamma_d(0)<1$ we have that $\Upsilon$ has a unique root. This root is a (global) maximum of $\Gamma_d$. Moreover, $\Gamma'_d>0$ to the left of $z_1$ and $\Gamma'_d<0$ to the right of $z_1$.
Finally note that $\Gamma_d(0)<1$ is equivalent to $\theta d - \alpha \phi(d) < \psi(d)$ or equivalently $\theta < \frac{\psi(d) + \alpha \phi(d)}{d} = \Theta_*(\alpha,d)$.

Now suppose $\Gamma_d(0)=1$ or equivalently $\theta = \frac{\psi(d) + \alpha \phi(d)}{d}= \Theta_*(\alpha,d)$.
Then $\Upsilon(0)=0$, but $\Upsilon'(0)= -h_\alpha'(0)<0$ so that $\Upsilon <0$ on a neighbourhood to the right of 0. Then, by similar arguments to the case $\Gamma_d(0)<1$ we find that $z_1=0$ and the unique positive root $z_2$ of $\Upsilon$ are turning points of $\Gamma_d$. Moreover $z_2$ is the maximizer.

Now suppose that $\Gamma_d(\frac{\alpha}{2}) > h_\alpha(\frac{\alpha}{2})$. Then working to the left of $\frac{\alpha}{2}$ we find that $\Upsilon>0$, and $\Gamma_d$ is decreasing and $h_\alpha$ is increasing on $[0,\frac{\alpha}{2})$, and there cannot be a root of $\Upsilon$ on $[0, \frac{\alpha}{2})$. Nor can $\Upsilon$ downcross 0 on $[\frac{\alpha}{2},\bar{z})$---here we use a similar argument to the case $\Gamma_d(0)=1$ to show that at any downcrossing $\Upsilon'>0$, a contradiction, but now on $[\frac{\alpha}{2},\bar{z})$ rather than 
$[z_1,\bar{z})$.
Since $\Upsilon >0$ we find $\Gamma_d$ is decreasing everywhere, and the maximum is at zero. The condition $\Gamma_d(\frac{\alpha}{2}) > h_\alpha(\frac{\alpha}{2})$ is equivalent to
$\theta > \Theta^*(\alpha,d)$. 

If $\Gamma_d(\frac{\alpha}{2}) = h_\alpha(\frac{\alpha}{2})$ then $\frac{\alpha}{2}$ is the unique root of $\Upsilon$. This follows since $\Gamma_d''(\frac{\alpha}{2})= 0>-2=h_\alpha''(\frac{\alpha}{2})$ so that $\Upsilon''(\frac{\alpha}{2})>0$ and $\Upsilon$ cannot hit zero either to the left or right of $\frac{\alpha}{2}$. Again, $\Gamma_d$ is decreasing and 0 is the maximiser.

Now suppose that $1 <\Gamma_d(0)$ and $\Gamma_d(\frac{\alpha}{2}) < h_\alpha(\frac{\alpha}{2})$. Then $\Upsilon(0)>0$ and $\Upsilon(\frac{\alpha}{2})<0$ and it follows that $\Upsilon$ has two roots $z_1,z_2$ with $0 < z_1< \frac{\alpha}{2} < z_2 < \bar{z}(\alpha)$. By similar arguments to the case $\Gamma_d(0) = h_\alpha(0)$ we see both that $z_1$ is the unique root of $\Upsilon$ in $(0,\frac{\alpha}{2})$ and that $z_2$ is the unique root of $\Upsilon$ in $(\frac{\alpha}{2},\bar{z})$. It follows that $\Gamma_d$ is decreasing below $z_1$ and above $z_2$ and increasing between $z_1$ and $z_2$ and thus $\Gamma_d$ attains it maximum either at zero or at $z_2$. Then $z_2$ is the maximiser if and only if $h(z_2) = \Gamma_d(z_2) \geq \Gamma_d(0)$ and 0 is the maximiser if and only if $h_\alpha(z_2) =  \Gamma_d(z_2) \leq \Gamma_d(0)$.

Let $K(\theta) = \Gamma_{d,\alpha,\theta}(0) - \Gamma_{d,\alpha,\theta}(z_2(d,\alpha,\theta))$.
Then $K(\Theta_*(\alpha,d))<0 <K(\Theta^*(\alpha,d))$. 
Then, since $\frac{\partial}{\partial z} \Gamma_{d,\alpha,\theta}|_{z=z_2} = 0$
\begin{eqnarray*} 
\frac{dK}{d \theta} & = & \frac{\partial }{\partial \theta} \Gamma_{d,\alpha,\theta}(0)
+ \frac{\partial }{\partial \theta} \Gamma_{d,\alpha,\theta}(z_2(\theta)) + \frac{\partial z_2}{\partial \theta} \frac{\partial}{\partial z_2} \Gamma_{d,\alpha,\theta}(z_2(\theta)) \\
& = &
\frac{d}{\psi(d)} - \frac{d \phi(z_2)}{\phi(z_2) \psi(d) + \psi(z_2) \phi(z)} > 0.
\end{eqnarray*}
It follows that $K$ is increasing in $\theta$ then we have that $K$ has a unique root $\hat{\Theta} \in (\Theta_*,\Theta^*)$. Moreover, it then follows that
$\Hat{\Theta}$ solves the implicit equation
\[ \hat{\Theta}(\alpha,d) = \frac{z_2 \psi(d)}{d \psi(z_2)} + \frac{\alpha}{d} \left( \phi(d) + \frac{(\phi(z_2)-1)}{\psi(z_2)} \psi(d) \right) 
\]
where $z_2 = z_2(\alpha,d,\theta)$.
\end{proof}

\begin{figure}[ht]
    \centering
    \includegraphics[width=0.85\linewidth]{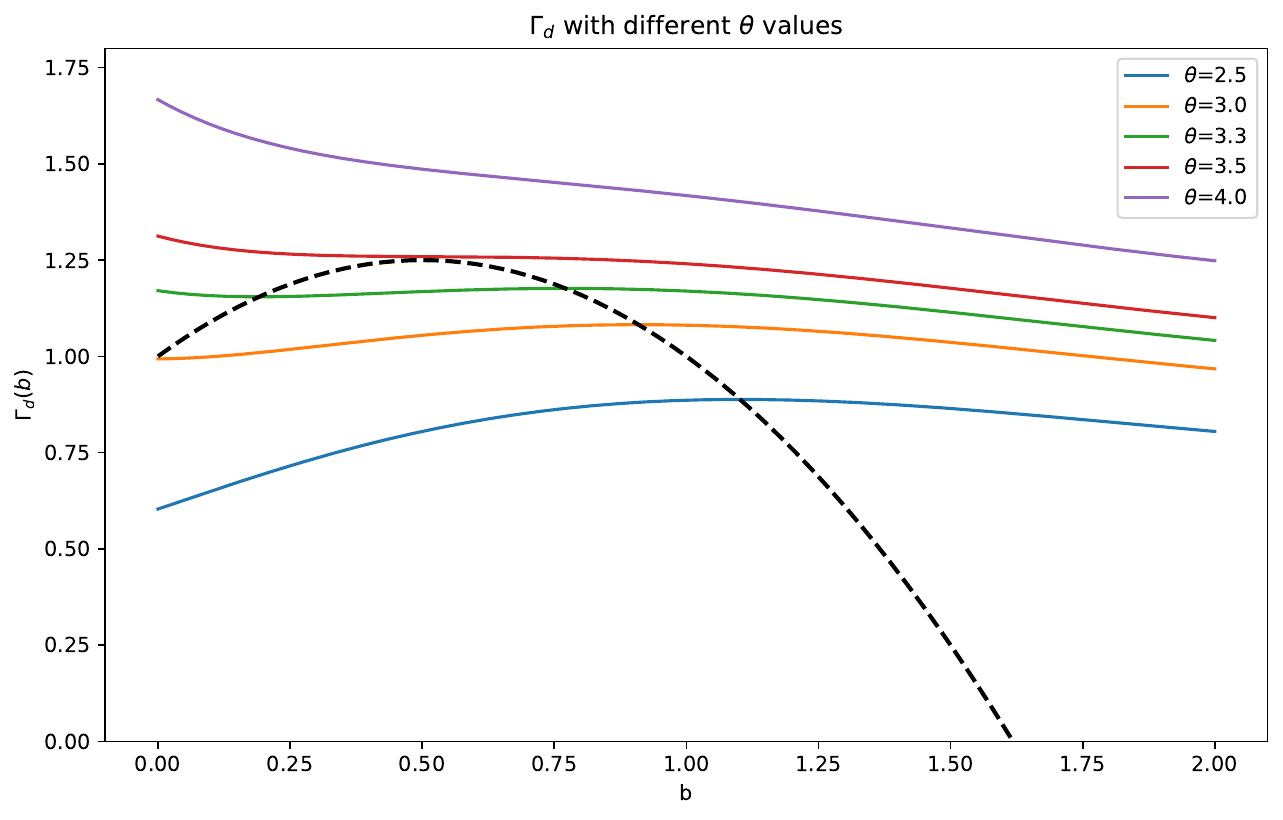}
    \caption{Examples of $\Gamma_{d}(\cdot) =\Gamma_{d,\theta}(\cdot)$ as $\theta$ varies. The curves $\Gamma_{d,\theta}$ are decreasing in $\theta$. Also shown with a dashed line is $h_{\alpha}$. The turning points of $\Gamma_d$ correspond to crossing points of $h_{\alpha}$ by $\Gamma_{d,\theta}$.}
    \label{Gamma_func_d}
\end{figure}

\begin{lemma}
\label{p0.5_maximiser_d=0}

Suppose $d=0$. Then $\Gamma_0$ has a unique maximiser at $z^*(0)$ where $z^*(0)$ is the unique solution in $(0,\infty)$ to $\psi(z) = \frac{z-\alpha}{1 - z(z-\alpha)}$. 
We have $z^*(0) \in (\alpha,\bar{z}) $ 
\end{lemma}

\begin{proof}
The result follows from the fact that $\Gamma_0(z) = \frac{(z-\alpha)}{\psi(z)}$ and that $\Gamma_0$ has a turning point at any $z$ for which $\psi(z) = \frac{(z-\alpha)}{1 - z(z-\alpha)}$. Since $\psi>0$ on $(0,\infty)$ and $\frac{z-\alpha}{1-z(z-\alpha)}>0$ on $(\alpha, \bar{z}(\alpha))$ only (on $(0,\infty)$ at least, and root to $\psi(z) = \frac{(z-\alpha)}{1 - z(z-\alpha)}$ must lie in $(\alpha, \bar{z}(\alpha))$.
Moreover, $(1-z(z-\alpha))\psi(z) - (z-\alpha)$ has derivative $-(z - \alpha)(1 + \psi(z)(z^2+1))$ and hence $(1-z(z-\alpha))\psi(z) - (z-\alpha)$ is strictly decreasing from $\psi(\alpha)$ at $\alpha$ to $-(\bar{z}(\alpha)-\alpha)$ at $\bar{z}$. It follows that there is exactly one solution to $\psi(z) = \frac{(z-\alpha)}{1 - z(z-\alpha)}$ in $(0,\infty)$.
\end{proof}

\begin{theorem}
\label{thm:symmetric2player}
Suppose that $p=0.5$, $\mathcal{G}_{1}$ and $\mathcal{L}_{1}$ are given by $\mathcal{G}_{1}(z)=z - \alpha$ and $\mathcal{L}_{1}(z)=  \theta z$, and that $\mathcal{G}_{2}$ and $\mathcal{L}_{2}$ are given by $\mathcal{G}_{2}(z)=-z - \alpha$ and $\mathcal{L}_{2}(z)=  -\theta z$.
 
The game admits at most three symmetric threshold-type Nash equilibria.

If $\theta \geq 1 + \max \{ 2 \alpha \left[ \alpha + (\alpha^2 + 1)\psi(\alpha) \right], \frac{1}{\alpha}(4+\alpha^2)\psi(\frac{\alpha}{2}) \}$ then there are no symmetric threshold-type Nash equilibria.

If $\theta \leq 1$ then there is a unique symmetric threshold-type Nash equilibria.

There are cases with more than one symmetric threshold-type Nash equilibria.
\end{theorem}

\begin{proof}
In Theorem \ref{Verification_theorem_2player} we showed that, at least if $\Gamma_d$ is unimodal, if Player 2 uses a fixed threshold strategy $\sigma^d = \inf \{ u : X_u \leq  -d\sqrt{1+u}\}$, then the optimal strategy for Player 1 is again a threshold strategy of the form $\tau = \inf \{ u : X_u \geq  b^*\sqrt{1+u} \}$, where $b^*=b^*(\alpha, \theta, d)$ is the maximiser of $\Gamma_d(\cdot)$. It follows that we expect symmetric equilibria to be solutions to $b^*(\alpha, \theta, d^*) = d^*$.
Since turning points satisfy $h_\alpha(z) = \Gamma_{d^*}(z)$, we wish to find solutions to $h_\alpha(z) = \Gamma_z(z)$. Solutions to this equation define candidates for the thresholds of symmetric threshold-type Nash equilibria. 
At this stage, they remain candidates because they correspond to turning points of $\Gamma_{z^*}$ and we need to check that they indeed define global maxima. 

Fix $\theta, \alpha>0$. Our first aim is to determine the number of roots of the equation 
$h_\alpha(z) = \tilde{\Gamma}(z)$,  where $\Tilde{\Gamma}:[0,\infty)\mapsto \mathbb{R}$ is given by
\begin{equation}\label{Symmetric_gamma_theta_general}
     \Tilde{\Gamma}(z)=\Gamma_{z}(z)=\frac{\mathcal{G}_{1}(z)-\mathcal{L}_{1}(-z)}{2\psi(z)}=\frac{(1+\theta)z-\alpha}{2e^{\frac{z^{2}}{2}}\int_{0}^{z}e^{-\frac{u^{2}}{2}}du}.
 \end{equation}
Then $h_\alpha(z) = \tilde{\Gamma}(z)$  
is equivalent to
$H_{\alpha,\theta}(z) = H(z; \alpha,\theta)=0$ where $H_{\alpha,\theta}:[0,\infty) \to \RR$ is given by
\begin{equation}\label{DefH}
    H_{\alpha,\theta}(z)=2(1-z^{2}+\alpha z)
    \psi(z) - (1+\theta)z+\alpha.
\end{equation}

Define $S_\alpha, T, U: [0,\infty) \to \RR$ by 
$S_{\alpha}(z)=2(\alpha-z)\Big[(z^{2}+1)\psi(z) +z\Big]$, $T(z) = (z^{2}+1)\psi(z) +z$ and $U(z) = z + \frac{T(z)}{T'(z)}$.
Then $H_{\alpha,\theta}'(z)= S_\alpha(z) + 1 - \theta$ and $H_{\alpha,\theta}''(z)= S'_\alpha(z)=2[(\alpha-z) T'(z) - T(z)]$. It follows that 
$S_\alpha'(z) = 0$ if and only if $U(z)=\alpha$.
We show that $U$ is strictly monotonic in $z$ and then it follows from $S_\alpha(0) = S_\alpha(\alpha)=0$ and $S_\alpha(\frac{\alpha}{2})>0$ that $S_{\alpha}(\cdot)$ is unimodal and hence $S_\alpha$ has a unique maximum.

To show that $U$ is increasing we use that $U' = 2 - \frac{T T''}{(T')^2}$. 
Using that $\psi'(z)=z\psi(z)+1$ we have that
$T'(z)=\psi(z)(z^{3}+3z)+z^{2}+2$ and $T''(z)= \psi(z)(z^{4}+6z^{2}+3)+z^{3}+5z$. Therefore, for all $z\geq 0$,
\begin{equation*}
    \begin{split}
        &2(T'(z))^{2}-T''(z)T(z) \\
       &\hspace{+2CM}= \psi(z)^{2}(z^{6}+5z^{4}+9z^{2}-3)+\psi(z)(2z^{5}+8z^{3}+16z)+z^{4}+3z^{2}+8 \\
      &\hspace{+2CM} > 3(z^4-1) \psi(z)^2 + 3z^2 + 8 > 0.
    \end{split}
\end{equation*}
where the last inequality follows immediately when $z \geq 1$ and, for $z<1$ from $\psi(z) < ze^{1/2} < z\sqrt{3}$ for $z \in (0,1)$. 
We conclude that $S_{\alpha}$ is unimodal. Indeed, since $S_\alpha(0)=0=S_\alpha(\alpha)$ and $S_\alpha(\cdot)>0$ on $(0,\alpha)$ we have that 
$S_\alpha$ has a unique maximum at $\hat{z}(\alpha)$ (say) where $\hat{z}(\alpha) \in (0,\alpha)$ is unique.

Note that $S_\alpha(z) < 2 \alpha \left[ \alpha + (\alpha^2 + 1)\psi(\alpha)\right]$ so that $\theta \geq 1+ 2 \alpha \left[ \alpha + (\alpha^2 + 1)\psi(\alpha)\right]$ is a sufficient condition for 
$\theta > 1 + S_\alpha(z)$ for all $z$.

\vspace{2mm}
Now we consider roots of $H_{\alpha,\theta}$ for different values of $\theta$.

\noindent{Case 1: Large $\theta$: $\theta \geq \theta^* := 1 + S_\alpha(\hat{z}(\alpha))$; there exists a unique root for $H_{\alpha,\theta}$.}

Recall that $H_{\alpha,\theta}'(z)=S_{\alpha}(z)+1-\theta.$ We have $H_{\alpha,\theta}'(z)\leq 0$ for every $z\geq 0$. Therefore, for fixed $\alpha$ and $\theta$, $H_{\alpha,\theta}(z)$ is monotonically decreasing with respect to $z$. Since $H_{\alpha,\theta}(0)=\alpha>0$, and $\lim_{z\to\infty}H(z)=-\infty$, there exists a unique root for $H_{\alpha,\theta}$.

\noindent{Case 2a: Moderate $\theta$: $\theta \in [1, \theta^* := 1 + S_\alpha(\hat{z}(\alpha)))$ and $H_{\alpha,\theta^*}(\hat{z}(\alpha)) \geq 0$; there exists a unique root for $H_{\alpha,\theta}$.}

From the definition of $H(z;\alpha,\theta)$ ,
we have that for fixed $z$, $H(z;\alpha,\theta)$ is decreasing in $\theta$.
Further, there are two roots of $H'_{\alpha,\theta}=0$ which we denote $z_{1}=z_1(\alpha,\theta)$ and $z_{2}=z_2(\alpha,\theta)$ chosen so that $z_{1}<\Hat{z}(\alpha)<z_{2}$. It is straightforward to see that $z_{1}$ is the local minimizer and $z_{2}$ is the local maximizer. 
Then, for any $z \leq z_2(\alpha,\theta)$,
\[ H(z; \alpha, \theta) \geq  H(z_1(\alpha,\theta); \alpha, \theta) > H(z_1(\alpha,\theta); \alpha, \theta^*) \geq H(\hat{z}(\alpha); \alpha, \theta^*) \geq 0, \]
where the first inequality follows from the fact that $z_1(\alpha)$ is a local minimiser of $H(\cdot; \alpha, \theta)$, the second from the strict monotonicity of $H$ in $\theta$, the third from the monotonicity of $H(z; \alpha, \theta^*)$ and the final inequality from the fact that we are in Case 2a. On the other hand, for any $z > z_2(\alpha,\theta)$, $H_{\alpha,\theta}'(z) <0$ and since $\lim_{z \rightarrow \infty} H_{\alpha,\theta}(z) = - \infty$ we conclude that there exists a unique root to $H_{\alpha,\theta}$.

\noindent{Case 2b: Moderate $\theta$: $\theta \in [1, \theta^* := 1 + S_\alpha(\hat{z}(\alpha)))$ and $H_{\alpha,\theta^*}(\hat{z}(\alpha)) < 0$; there exists at least one and at most three roots for $H_{\alpha,\theta}$.}

$H_{\alpha,\theta}$ is initially decreasing, then increasing and then decreasing again. Hence it can cross zero at most three times.

\noindent{Case 3: Small $\theta$: $0 < \theta \leq 1$; there exists a unique root of $H_{\alpha,\theta}$.}

Suppose that $\theta=1$. Then, $H_{\alpha,\theta}'(z) = S_\alpha(z)$ and for $z<\alpha$ we have $H_{\alpha,\theta}'(z)>0$ and for $z>\alpha$ we have $H_{\alpha,\theta}'(z)<0$. Since $H_{\alpha,\theta}(0)=\alpha>0$ and $\lim_{z \rightarrow \infty} H_{\alpha,\theta}(z) = -\infty$ there is a unique root to $H_{\alpha,\theta}(\cdot) = 0$ in $(\alpha,\infty)$.

Now suppose that $\theta<1$. Then $H(z;\alpha, \theta)> H(z; \alpha,1)$ and there is no root of $H_{\alpha,\theta}$ for $z \leq \alpha$. $H_{\alpha,\theta}$ reaches a positive maximum at $\hat{z}(\alpha)$ after which it is decreasing, and therefore must have a unique root in $(\alpha,\infty)$.

\vspace{2mm}

The roots of $H_{\alpha,\theta}$ are candidates for the threshold values of symmetric threshold-type Nash equilibria. In addition, zero is a potential threshold value for a symmetric threshold-type Nash equilibrium (the roots of $H_{\alpha,\theta}$ corresponding to turning points, but we might also have a maxima at an extremal point). However, we can rule out $0$ as generating a symmetric threshold-type Nash equilibria as $\Gamma_0(z) = \frac{(z - \alpha)}{\psi(z)}$ and this is not maximised at zero. Thus, if Player 2 stops at a threshold of zero, then Player 1 should use a strictly positive threshold in response.

In order
to show that a candidate thresholds (i.e. a root $z^*$ of $H_{\alpha,\theta}$) generates a symmetric threshold-type Nash equilibrium solution
we want to apply Theorem~\ref{Verification_theorem_2player}.
Assumption~\ref{supermartingale_assumption} is satisfied if $z^*> \alpha/2$. Further, Assumption~\ref{ass:Gamma} is saisfied if
$\Gamma_{z^*}= (\Gamma_{z^*}(z))_{z \geq 0}$ is maximized at $z=z^*$ (all we currently know is that it is a turning point) and if $\Gamma_{z^*}'<0$ for $z>z^*$. Thus we need 
\begin{enumerate}
     \item[(a)] $z^* \geq \frac{\alpha}{2}$;
     \item[(b)]   $\Gamma_{z^*}(z^*)\geq\Gamma_{z^*}(0)$.
\end{enumerate}
(Note that if $z< \frac{\alpha}{2}$ is a turning point of $\Gamma_{z^*}$ then $z$ is a local minimum. Otherwise, if $z^* > \alpha/2$ then $z^*$ is an upcrossing of zero by $\Upsilon_{z^*} = \Gamma_{z^*} - h$.)

It is easily seen that $\Gamma_d(0) < \Gamma_d(d)$ if and only if $\theta < 1 + \alpha \frac{2 \phi(d) - 1}{d}$. Let $\Psi:(0,\infty) \to (0,\infty)$ be given by $\Psi(d) = \frac{2 \phi(d) - 1}{d}$. It can be verified that $\Psi$ has a minimum value $\Psi_* = \frac{2 \phi(d_*) - 1}{d_*}$ attained at $d_*$ where $d_*$ is the unique root in $(0,1)$ of $2(1-d^2)\phi(d) - 1=0$. (Uniqueness follows from the fact that $(1-d^2)\phi(d)$ is decreasing.). We find that numerically $d_{*}\approx 0.79765$ and $\Psi_*\approx 2.1928$.

Then, a sufficient condition for $\Gamma_{d}(0) < \Gamma_{d}(d)$ for all $d$ including at $d=d^*$ is that $\theta < 1 + \Psi_* \alpha$. 

Suppose $\theta \leq 1$ so that we are in Case 3. Then there is a unique root of $H_{\alpha,\theta}$, and this root $z^*$ satisfies $z^* \geq \alpha > \frac{\alpha}{2}$; moreover $\theta < 1 + \Psi_* \alpha$. It follows that both the conditions (a) and (b) above are satisfied and that $z^*$ defines a symmetric threshold-type Nash equilibrium. See Figure~\ref{UniqueSolution_plot}.

 Suppose now that $\theta \geq \theta^* = 1 + S_\alpha(\hat{z}(\alpha))$. 
Then $H_{\alpha,\theta}$ has a unique root $z^*$; this root is the only candidate for a symmetric threshold-type Nash equilibrium. A sufficient condition for (a) to {\em not} hold is that $H_{\alpha,\theta}(\frac{\alpha}{2})<0$ or equivalently $\theta>\frac{4}{\alpha}(1+\frac{\alpha^{2}}{4})\psi(\frac{\alpha}{2})+1$. Recall that $S_{\alpha}(z)<2\alpha[\alpha+(\alpha^2 +1)\psi(\alpha)]$. It follows that if $\theta> 1+ \max \{2\alpha[\alpha+(\alpha^2 +1)\psi(\alpha)], \frac{1}{\alpha}(4+\alpha^{2})\psi(\frac{\alpha}{2}) \}$ then there is only one non-zero candidate threshold $z^*$ which might define a symmetric threshold-type Nash equilibrium. However, this threshold is a local minimizer of $\Gamma_{z^*}(\cdot)$, thus it cannot be the optimal response. It follows that in this case no symmetric threshold-type Nash equilibrium exists. See Figure~\ref{NoSolution_plot}.

Now we consider moderate $\theta$. In Case 2b above there can be multiple candidate symmetric threshold-type Nash equilibria, in the sense that there are values of $\alpha$ and $\theta$ for which $H_{\alpha,\theta}$ has multiple roots (at most three). We give a numerical example to show that it is possible that there are two roots which define symmetric threshold-type Nash equilibria. To do this we need to check that for two candidates (a) and (b) above hold.
See Figure~\ref{TWOSolution_plot}.

\end{proof}

\begin{figure}[ht]
    \centering

    \begin{subfigure}[t]{0.49\textwidth}
        \vspace{0pt}
        \centering
        \includegraphics[width=\linewidth]{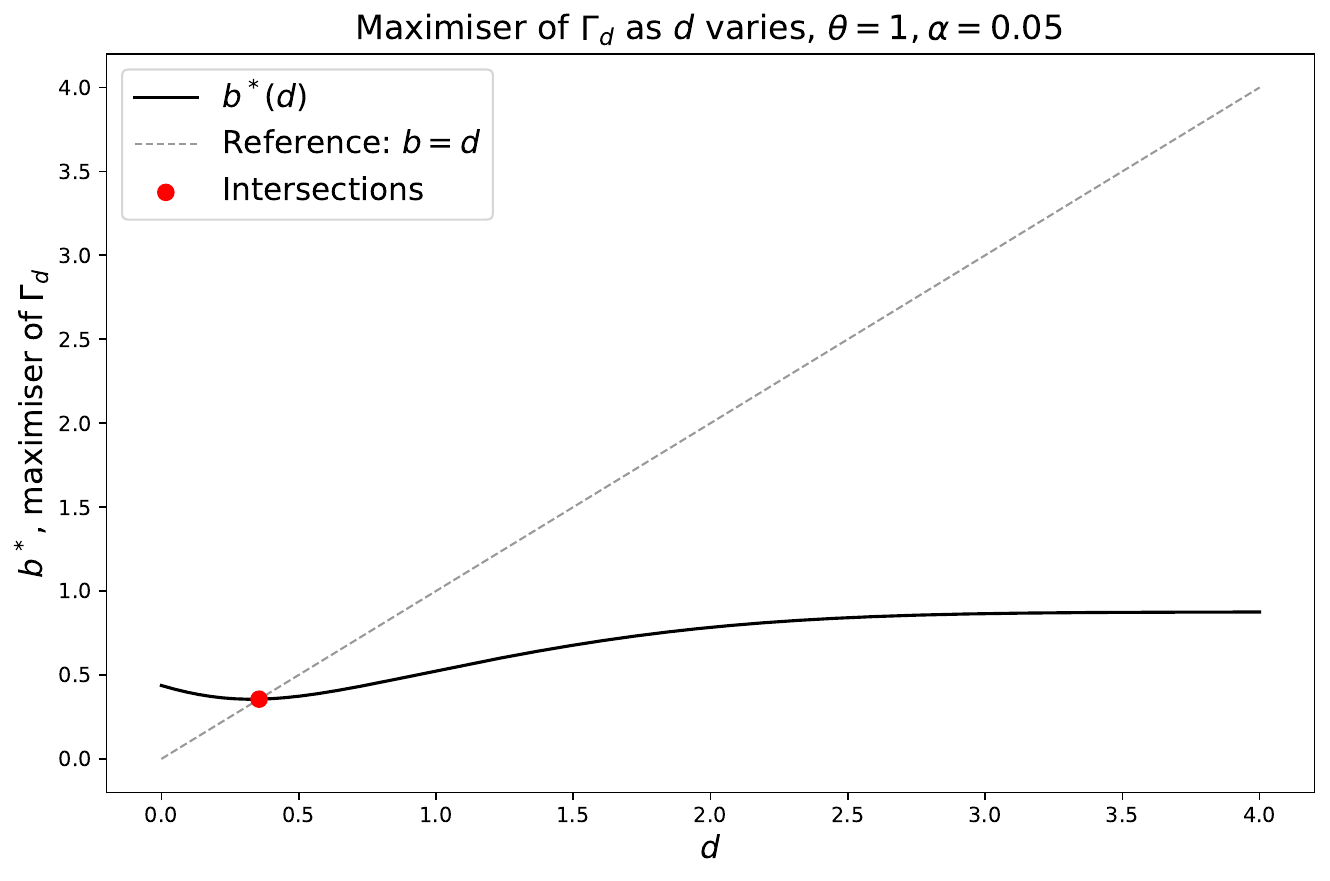}
        \caption{An example where there is a unique symmetric Nash equilibrium. There is a symmetric threshold-type Nash equilibrium where the maximiser $b^*(d)$ of $\Gamma_d$ crosses the diagonal.}
        \label{UniqueSolution_plot}
    \end{subfigure}
    \hfill
    \begin{subfigure}[t]{0.49\textwidth}
        \vspace{0pt}
        \centering
        \includegraphics[width=\linewidth]{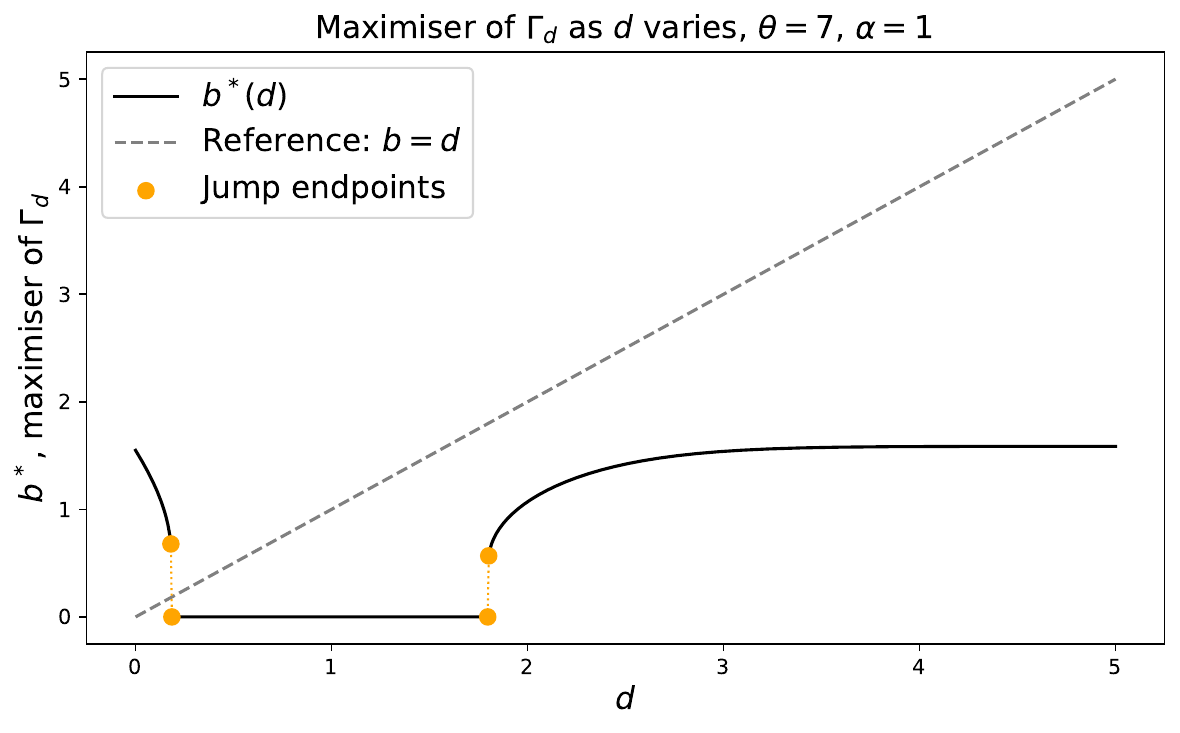}
        \caption{An example where there is no symmetric threshold-type Nash equilibrium.}
        \label{NoSolution_plot}
    \end{subfigure}

    \caption{Examples illustrating the unique existence and non-existence of symmetric threshold-type Nash equilibria.}
    \label{fig:comparison}
\end{figure}

\begin{figure}[ht]
    \centering
    \includegraphics[width=0.8\linewidth]{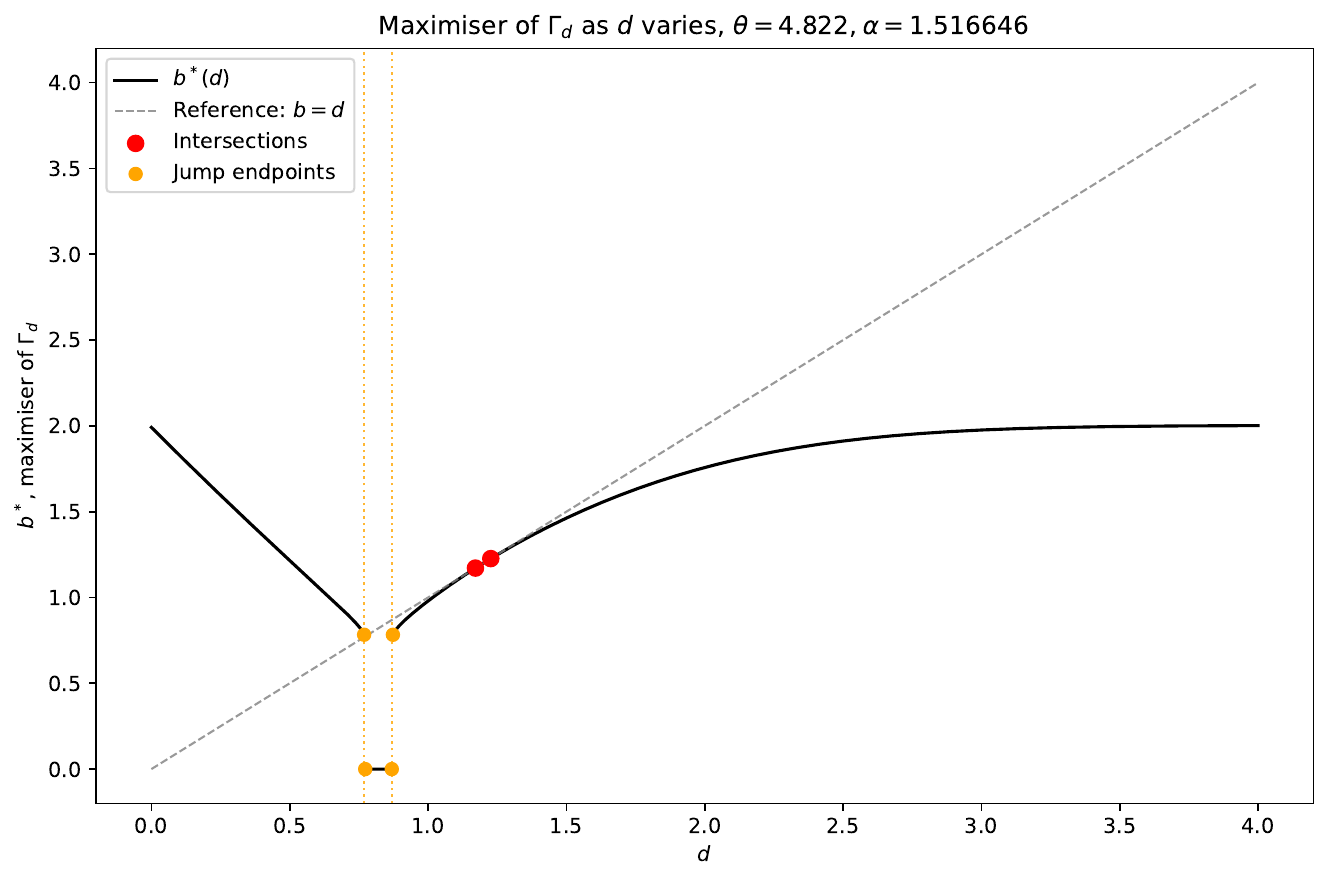}
    \caption{An example where there are two symmetric threshold-type Nash equilibrium. The conditions $z^*\geq\frac{\alpha}{2}$ and $\Gamma_{z^*}(z^*)\geq \Gamma_{z^*}(0)$ are satisfied numerically for both candidate thresholds $d=1.18$ and $d=1.24$.}
    \label{TWOSolution_plot}
\end{figure}

\section{The $N$-player game with a spider Brownian motion}
 \label{sec:spidergame}

\subsection{Game formulation}
\label{ssec:gameformulation}

Intuitively speaking, a spider Brownian motion is a generalized version of a skew Brownian motion constructed on a star-shaped graph in the plane. The state-space consists of $N$ half-lines, or `rays', glued together at a common point (the vertex). When the spider process is on a given ray it behaves like a standard Brownian motion. Since a Brownian motion always returns to 0, the spider process always returns to the vertex. At the vertex, the spider process chooses one of the rays randomly according to fixed probabilities, and having chosen a ray the process then proceeds as a Brownian motion on that ray. 

Given $N \in \NN$ with $N\geq 2$ let $I^\Theta = \{ \theta_1, \theta_2, \ldots, \theta_N \}$ be a set of distinct non-zero elements (typically interpreted as angles in the unit circle). Let $X = (X_t)_{t \geq 0}$ denote a Brownian motion and let $R =(R_t)_{t \geq 0}=(|X_t|)_{t \geq 0}$ denote the modulus of $X$ so that $R$ is a reflecting Brownian motion. Then $X$ (or $R$) can be decomposed into excursions away from zero. The set of such excursions is countable and can therefore be indexed by the set of natural numbers. Let $I_k=(g_k,d_k)$ be the $k^{th}$ excursion interval under this indexation and then $R_{g_k} = 0 = R_{d_k}$ and $R > 0$ on $I_k$.

Define the function $s : [0,1] \to \NN \cup \{0\}$ by $s(t) = k$ if $R_{t} \neq 0$ and $t \in (g_k,d_k)$ and $s(t)=0$ if $R_t=0$. Let $(\Theta^\NN_k)_{k \geq 1}$ be a sequence of independent, identically distributed, $I^{\Theta}$-valued random variables which is independent of $X$ and such that
\[ \PP( \Theta^{\NN}_k = \theta_j) = p_j  \hspace{10mm} k\geq 1, 1 \leq j \leq N. \]
Here $(p_1, p_2, \ldots, p_N)$ is a strictly positive probability vector, i.e $p_j>0$ for each $j \in \{1,2, \ldots , N \}$ and $\sum_{j=1}^N p_j = 1$. Let $(\Theta_t)_{0 \leq t \leq 1} = \Theta^{\NN}_{s(t)}$.
Then $\Theta$ is a $I^\Theta \cup \{0 \}$-valued process which is constant on excursions of the Brownian motion away from zero, and the values on different excursions are independent.

The spider Brownian motion is the process $(R,\Theta)= (R_t, \Theta_t)_{t \geq 0}$.
The state space of $(R, \Theta)$ is $E = ( (0,\infty) \times \{\theta_1,\theta_2, \ldots \theta_N \} ) \cup \{0,0\}$.
We call $R$ the radial component and $\Theta$ the angular component.

We suppose that the probability space $(\Omega, \mathcal{F}, \mathbb{P}, \mathbb{F}=(F_{t})_{t\geq 0})$ satisfies the usual conditions and supports the reflecting Brownian motion $R$ and the angular process $\Theta$. In particular, we suppose that
\[ \sF_t = \sigma \left( (R_s; 0 \leq s \leq t) \cup (\Theta^\NN_k; g_k \leq t) \right). \]

Now we introduce a generalisation of the two-player non-zero-sum optimal stopping game presented in the previous section to a $N$-player game. In place of a skew-Brownian motion the underlying stochastic process is the spider Brownian motion $(R, \Theta)$. The generalized game involves $N$ players, labeled Player 1, Player 2,... up to Player $N$. Player $j$ chooses a stopping time $\tau_{j} \in \Tau^j$ where $\Tau^j$ is the set of stopping times such that either $\Theta_{\tau_j} = \theta_j$ or $\tau_j=1$ (i.e. Player $j$ can only stop the game when the angular part of the spider Brownian motion is in $\{ \theta_j \}$, or at the terminal time). If Player $k$ stops the game first (i.e. if $\tau_k < \min_{\ell \neq k} \tau_\ell$) then Player $j$ receives $\GG^j_k(\tau_k, R_{\tau_k} , \Theta_{\tau_k}) = \GG^j_k(\tau_k, R_{\tau_k} , \theta_k)$ (and if the game is not stopped before the terminal time then no player receives a payoff, i.e. $\GG^j(1,0,0)=0$). For a family of stopping times $(\tau_1, \tau_2, \ldots \tau_N)$ with $\tau_j \in \Tau^j$ the payoff $P^j$ to Player $j$ is given by
\begin{equation}
\label{eq:payoff}
        P^{j}(\tau_{1},...,\tau_{N})=
        \sum_{k=1}^{N}\GG^j_{k}(\tau_{k},R_{\tau_{k}},\Theta_{\tau_k}) \mathbb{I}_{\{\tau_{k}< \min_{j \neq k} \tau_{j}\}}.
\end{equation}

If the payoffs have certain scaling properties then we expect that these scaling properties will be inherited by the game, and then we may expect that the players each play threshold strategies of square root form. To this end we assume that for each $j,k \in \{ 1, \ldots , N \}$ we have
\[ \GG_k^j(t,r,\theta_j) = \frac{1}{\sqrt{1+t}} \sG_k^j \left( \frac{r}{\sqrt{1+t}} \right) \]
for a family of continuous, twice-differentiable functions $(\sG_k^j)_{1 \leq j,k \leq N} : (0,\infty) \mapsto \RR$ which are bounded below. In the typical example we have that for $j \neq k$ we have $\sG_k^j : (0,\infty) \to \RR$ is negative and decreasing (corresponding to the fact that if Player $k$ stops the game when $\Theta_t = \theta_k$ then Player $j$ makes a payment to Player $k$, and this payment is increasing in the magnitude of $R$). Also, in the typical example we have that for $j = k$ we have $\sG_j^j : (0,\infty) \to \RR$ is increasing (corresponding to the fact that if Player $j$ stops the game when $\Theta_t = \theta_j$ then the payment Player $j$ receives is increasing in the magnitude of $R$). Our canonical example is $\sG_k^j(z) = Q_k^j z - \alpha_j \delta_{jk}$ where $(Q_k^j)_{1 \leq j,k \leq 1}$ is a constant matrix in $\RR^N \times \RR^N$ with $Q_j^j>0$ and $Q_k^j<0$, $\delta_{jk}=1$ if $j=k$ and zero otherwise and $\alpha = (\alpha_j)_{1 \leq j \leq N}$ is a constant vector in $(0,\infty)^N$.

If $\sum_{k=1}^N Q_k^j = 0$ then we can think of the game as a zero sum game modified by a cost paid by the stopper.

\subsection{An optimal stopping problem for Player $i$}
\label{ssec:gameosp}

Fix $i \in \{1, \ldots n \}$. We consider the game from the perspective of Player $i$. Suppose that the other players are playing threshold strategies of square-root form so that for $j \neq i$, Player $j$ uses the stopping rule $\tau_j = \inf\{  u : \Theta_u = \theta_j , R_u \geq d_j \sqrt{1+u} \} $ where the threshold constants $(d_j)_{j \neq i}$ are given. Then the problem from Player $i$'s perspective is an optimal stopping problem.

The ansatz is that
\begin{enumerate}
\item Player $i$ will also play a threshold strategy of square root form: $\tau_i = \inf\{  u : \Theta_u = \theta_i , R_u \geq b \sqrt{1+u} \}$, where $b$ is to be determined.
\item The problem can be rewritten in terms of $(Z_t, \Theta_t)_{0 \leq t < \infty}$ where $Z_t = \frac{R_t}{\sqrt{1+t}}$.
\item The value function $V^i$ for Player $i$ will be of the form $V^i(t,r,\theta) = \sum_{j=1}^{N}\frac{1}{\sqrt{1+t}} f^i_j \left( \frac{r}{\sqrt{1+t}} \right) \mathbb{I}_{\{\theta = \theta_j\}} + \frac{1}{\sqrt{1+t}} f^i_0 \mathbb{I}_{\{\theta = 0\}}$ where $f^i_j$ is continuous on $(0,\infty)$, including at $d_j$ for $j \neq i$ and at $b$ for $j=i$.
\item On the stopping region, for $j \neq i$, $f^i_j(z) = \sG_j^i(z)$ for $z \geq d_j$, and for $j=i$, $f^i_i(z) = \sG_i^i(z)$ for $z \geq b$.
\item The value function is continuous at 0 so that $\lim_{z \downarrow 0}f^i_j(z)$ does not depend on $j$ and is equal to $f^i_0$.
\item The $p$-weighted sum of the radial derivatives at 0 is zero so that $\sum_{j=1}^N p_j \frac{\partial}{\partial z} f^i_j (z) = 0$.
\item On the continuation region, scalings of the problem mean that $f^i_j=f^i_j(z)$ solves the ordinary differential equation $f'' - zf' - f=0$.
\end{enumerate}
The above elements of the ansatz apply to the case where Player $i$ follows any square-root threshold strategy.
The final part of the ansatz is that for the optimal stopping rule 
\begin{enumerate}
\item[8.] The optimal value $b^*$ of $b$ is determined so that the value function is maximised uniformly over the state space and that this is equivalent to smooth fit of the value function at $b$.
\end{enumerate}

Recall that the fundamental solutions of $f'' - zf' - f=0$ are $\phi$ and $\psi$. Then from the ansatz we expect that
\[ f_j^i(z) = A_j^i \phi(z) + B_j^i \psi(z) \]
where the constants $(A_j^i)_{j = 1,2,\ldots, N }$ and $(B_j^i)_{j = 1,2,\ldots, N }$ are to be determined. Note that $A_j^i$ and $B_j^i$ will depend on $b$ so sometimes we will write them as functions of $b$, but they are constants in that they do not depend on $z = \frac{r}{\sqrt{1+t}}$.

By continuity at 0 we have that $A_j^i=A^i$. For $j \neq i$, by value matching at $d_j$ we have that $\sG_j^i(d_j) = A^i \phi(d_j) + B_j^i \psi(d_j)$ and by value matching at $b$ we have
\begin{equation}
\label{eq:ABvaluematching}
\sG_i^i(b) = A^i \phi(b) + B_i^i \psi(b).
\end{equation}
By the derivative condition at 0 we also have $\sum_{j=1}^N p_j B_j^i = 0$.

It follows that for each $j \neq i$ we have that $A^i$ solves $A^i \frac{\phi(d_j)}{\psi(d_j)} + B^i_j  = \frac{\sG_j^i(d_j)}{\psi(d_j)}$, with a similar expression for $j=i$ but replacing $d_j$ with $b$. Then multiplying by $p_j$ and summing over $j$ we have that
\[ A^i = A^i(b) = \frac{\frac{p_i \sG_i^i(b)}{\psi(b) } + \sum_{j=1,j \neq i}^N \frac{p_j \sG^i_j(d_j)}{\psi(d_j)}}{\frac{p_i \phi(b)}{\psi(b)} + \sum_{j=1, j \neq i}^N \frac{p_j \phi(d_j)}{\psi(d_j)}} = \frac{p_i \sG_i^i(b)+ {\psi(b)} \sD_{-i}}{ p_i \phi(b) + {\psi(b)}  \sE_{-i} }  \]
where $\sE_{-i} = \Big\{\sum_{j \neq i, j=1}^{N} \frac{p_{j}\phi(d_{j})}{\psi(d_j)} \Big\}$ and $\sD_{-i} = \Big\{\sum_{j \neq i,j=1}^{N} \frac{p_{j}\sG^i_{j}(d_{j})}{\psi(d_j)} \Big\}$ depend on $(d_j)_{j \neq i}$ but do not depend on $b$.
Then we have
\begin{equation}
\label{eq:defBii} B^i_i = B_i^i(b) = \frac{\sG_i^i(b)}{\psi(b)} - A^i  \frac{\phi(b)}{\psi(b)} = \frac{ \sE_{-i} \sG^i_i(b) - \sD_{-i} \phi(b)}{\sE_{-i} \psi(b) + p_i \phi(b)}
\end{equation}
and hence also that $A^i \sE_{-i} - \sD_{-i} = p_i \frac{\sE_{-i} \sG_i^i - \sD_{-i} \phi(b)}{\sE_{-i} \psi(b) + p_i \phi(b)} = p_i B^i_i$ so that
\begin{equation}
\label{eq:defAi}
   A^i = A^i(b) =  \frac{\sD_{-i}}{\sE_{-i}} + \frac{p_i}{\sE_{-i}} B_i^i(b).
\end{equation}
We also have
\begin{equation}
\label{eq:defBij}
B_j^i = \frac{\sG_j^i(d_j)}{\psi(d_j)} - A^i  \frac{\phi(d_j)}{\psi(d_j)}
= \frac{ \sE_{-i} \sG_j^i(d_j) - \sD_{-i} \phi(d_j)}{\sE_{-i} \psi(d_j)} - \frac{\phi(d_j)}{\psi(d_j)} \frac{p_i}{\sE_{-i}} B_i^i(b).
\end{equation}

Recall that we have fixed Player $i$, and the square root thresholds of the other players as represented by the family $(d_j)_{1 \leq j \leq N, j \neq i}$. To ease notation we temporarily drop the superscript $i$ for Player $i$ and write $\sE = \sE_{-i}$ and $\sD = \sD_{-i}$, together with $A$, $(B_j)_{1 \leq j \leq n}$, $(\sG_j)_{1 \leq j \leq N}$ in place of $A^i$, $(B_j^i)_{1 \leq j \leq N}$ and $(\sG_j^i)_{1 \leq j \leq N}$. Note that $\sE$, $\sD$ and $(\sG_j)_{1 \leq j \leq N, j \neq i}$ do not depend on $b$, but $A$, $(B_j)_{1 \leq j \leq N}$ and $\sG_i$ do depend on $b$.

We consider the scaled process $(Z_t,\Theta_t)_{0 \leq t \leq 1}$ where for $t<1$, $Z_t = \frac{R_t}{\sqrt{1+t}}$ and $Z_1 = 0$. The state space of $(Z,\Theta)$ is the same as the state space $E$ of $(R,\Theta)$.
Introduce the function $f = f^b : E = ((0,\infty) \times ( \{\theta_1,\theta_2, \ldots \theta_N \} ) \cup \{0,0\} \to \RR$ by
\begin{equation}
\label{eq:deffb}
f^b(z,\theta) = \begin{cases}  A(b)    & (z,\theta)=(0,0)  \\
                               A(b) \phi(z) + B_i(b) \psi(z)     & 0< z < b, \theta = \theta_i \\
                                \sG_i(z)   & z \geq b, \theta = \theta_i \\
                                A(b) \phi(z) + B_j(b) \psi(z)     & 0< z < d_j, \theta = \theta_j, 1 \leq j \leq N, j \neq i\\
                                \sG_j(z)   & z \geq d_j, \theta = \theta_j, 1 \leq j \leq N, j \neq i 
                                \end{cases}   \end{equation}
We write $f'(z, \theta_j)$ for $\frac{\partial f^b}{\partial z}|_{(z,\theta_j)}$. By construction, $f^b(\cdot,\theta_i)$ is continuous at $b$ (recall \eqref{eq:ABvaluematching}) and $f^b(\cdot,\theta_j)$ is continuous at $d_j$.

Define $\Gamma:(0,\infty) \to \RR$ by
\[ \Gamma(b) = B_i = \frac{ \sE \sG_i(b) - \sD \phi(b)}{\sE \psi(b) + p_i \phi(b)}   \]
Clearly $\Gamma$ is continuous. In general we have that $\Gamma$ has a supremum but this supremum may not be attained. In the next assumption we suppose that the payoff structure is such that the supremum is attained, at $b^*$ say where $b^* \in (0,\infty)$.

\begin{assumption}
\label{ass:Gamma_N}
$\Gamma$ has a maximum, at $b^*$ say, where $b^*>0$, and $\Gamma$ is non-increasing on $[b^*,\infty)$.
\end{assumption}

\begin{prop} Suppose Assumption~\ref{ass:Gamma_N} holds.

Then $f^{b}$ given in \eqref{eq:deffb} attains its maximum over $b$ at $b^*$, uniformly in $(z,\theta) \in E$.
\end{prop}

\begin{proof}
The proof follows exactly as in the proof of Proposition~\ref{prop:maximiser}.

\end{proof}

\begin{corollary}\label{cor:valuebddpayoffs}
 For any $z\in (0,b^*)$ we have $f^{b^*}(z,\theta_i)\geq \mathcal{G}_i(z)$. 
\end{corollary}
\begin{proof}
This follows as in the proof of Corollary~\ref{Cor:valuebddpayoffs}.

\end{proof}

\begin{lemma}
\label{lem:smoothfit}
Suppose Assumption~\ref{ass:Gamma_N} holds. Then $f^{b^*}$ is continuously differentiable at $(b^*,\theta_i)$.
\end{lemma}

\begin{proof} Continuity of $f^{b^*}$ at $b^*$ follows from \eqref{eq:ABvaluematching} and \eqref{eq:deffb}.
For the derivative we have
\[ \Gamma'(b) = \frac{\sE}{p \phi(b) + \sE \psi(b)} \left( \sG_i'(b) - b \sG_i(b) - \Gamma(b) \right) \]
Then, at $b^*$ we must have $\sG_i'(b^*) = b^* \sG(b^*) +  \Gamma(b^*)$.

Then, for $z<b^*$, $f'(z,\theta_i) = A(b^*) z \phi(z) + B_i(b^*) (z \psi(z) + 1) = zf(z) + B_i(b^*)$ and $f'(b^*-, \theta_i) = b^*f(b^*) + B_i(b^*) = b^* \sG_i(b^*) + \Gamma(b^*) = \sG_i'(b^*) = f'(b^*+,\theta_i)$.
\end{proof}

Based on our ansatz we define a candidate value function for Player $i$:
\begin{equation}\label{Expression_VE_star}
    V_{*}(t,r, \theta)=
    \begin{cases}
        \frac{1}{\sqrt{1+t}}A(b^*)  &  (r, \theta) =(0,0) \\
        \frac{1}{\sqrt{1+t}}\left[A(b^*)\phi(\frac{r}{\sqrt{1-t}})+B_i(b^*,d)\psi(\frac{r}{\sqrt{1+t}})\right],
         &  \theta = \theta_i , \ 0<  r < b^*\sqrt{1+t},\\
        \mathbb{G}_{i}(t,r, \theta_i),  & \theta = \theta_i , r \geq b^*\sqrt{1+t}  \\
        \frac{1}{\sqrt{1+t}}\left[A(b^*)\phi(\frac{r}{\sqrt{1+t}})+B_j(b^*,d)\psi(\frac{r}{\sqrt{1+t}})\right],
         &  \theta = \theta_j ,  0<  r < d_j \sqrt{1+t},  \\ 
        \mathbb{G}_{j}(t,r, \theta_j),  & \theta = \theta_j ,  r \geq  d_j \sqrt{1+t}, \\ 
     \end{cases}
\end{equation}
where $A(b^*), B_i(b^*)$ and $B_j(b^*)$ are the functions $A$, $B_i$ and $B_j$ given in \eqref{eq:defAi}, \eqref{eq:defBii} and \eqref{eq:defBij} and evaluated at $b^*$ and the last two lines hold for $1 \leq j \leq N, j \neq i$.

The problem facing Player $i$ in the game started at $(t_0,r,\theta)$ (given the strategies of the other players, and the fact that these are threshold strategies of square-root form) is to find
\begin{equation}\label{Valuefunc_P1_E}
    \begin{split}
      \overline{V}(t_0,r,\theta) &= \sup_{\tau \in \sT^{i,t_0}}
 \EE_{t_0,r,\theta} \Bigg[ \GG_i(\tau,R_{\tau},\Theta_\tau)
\mathbb{I}_{\{\tau < \min_{j \neq i} \tau^{t_0}_j \} }\\
    &\hspace{+5cm}+ \sum_{j \neq i} \GG_{j}(\tau^{t_0}_j, R_{\tau^{t_0}_j}, \Theta_{\tau^{t_0}_j}) \mathbb{I}_{\{ {\tau^{t_0}_j} < \tau \wedge \min_{k \neq j,i} {\tau^{t_0}_k}\}}
        \Bigg]   
    \end{split}
\end{equation}
where for $j \neq i$, ${\tau^{t_0}_j} = \inf \{ u \geq t_0; \Theta_t = \theta_j, R_t \geq d_j \sqrt{1+u} \}$, and $\sT^{i,t_0}$ is the set of stopping times $\tau$ such that $\Theta_\tau = \theta_i$ whenever $t_0 \leq \tau < \infty$.

It remains to show that the candidate function $V_{*}$ in \eqref{Expression_VE_star}
is indeed the solution to the value function $\overline{V}$ defined in \eqref{Valuefunc_P1_E}.

\begin{lemma}

\begin{enumerate}
\item $M^\phi= (M^{\phi}_t)_{t\geq 0}$ given by $M^{\phi}_t = \frac{1}{\sqrt{1+t}} \phi\left( \frac{R_t}{\sqrt{1+t}} \right)$ is a non-negative local martingale. 
\item For $j \in \{ 1,2, \ldots , N \}$, $M^{\psi,j} =  (M^{\psi,j}_t)_{t\geq 0}$ given by $$M^{\psi,j}_t = \frac{1}{\sqrt{1+t}} \psi\left( \frac{R_t}{\sqrt{1+t}} \right) \left( \mathbb{I}_{ \{ \Theta_t = \theta_j \} } - p_j \right)$$ is a local martingale.
\item Suppose $C \in \RR$ and $D \in \RR^N$ is such that $\sum_{j=1}^N p_j D_{j}=0$. Then $M = (M^{C,D}_t)_{0 \leq t \leq 1}$ given by
\[ M^{C,D}_t = C \frac{1}{\sqrt{1+t}} \phi\left( \frac{R_t}{\sqrt{1+t}} \right) + \sum_{j=1}^N D_j \frac{1}{\sqrt{1+t}} \psi\left( \frac{R_t}{\sqrt{1+t}} \right) \mathbb{I}_{ \{ \Theta_t = \theta_j \} }
 \]
is a local martingale.
\end{enumerate}
\end{lemma}

\begin{proof}
First we prove (i). Since $R$ is a reflecting Brownian motion, $dR_t = dB_t  + dL^R_t$, and
\[ dM_t = \phi'\left( \frac{R_t}{\sqrt{1+t}} \right) (dB_t + dL^R_t)  \]
where the $dt$ term disappears since $\phi$ solves $\phi'' - z \phi' - \phi = 0$.
But, $L^R_t$ only grows when $R_t=0$ and $\phi'(0)=0$ so that $\phi'( \frac{R_t}{\sqrt{1+t}}) dL^R_t = 0$. Hence $M$ is a local martingale.

Now we prove (ii). Suppose $(R_{t_0},\theta_{t_0}) = (r, \theta_i)$. Let $T_{0}^{t_0} = \inf \{ t \geq t_0 : R_t = 0 \}$. For $t_0 \leq t \leq T_{0}^{t_0}$ a similar argument to the case for (i) (combined with the fact that $\Theta_t = \theta_i$ is constant for $t \leq T_{0}^{t_0}$) gives that $(M^{\psi,j}_t)_{t_0 \leq t \leq T_{0}^{t_0}}$ is a non-negative local martingale. In particular, for a reducing sequence $(S_n)_{n \geq 1}$ such that $R_{t \wedge S_n}$ is bounded,
\[ \EE[ M^{\psi,j}_{t \wedge T_{0}^{t_0} \wedge S_n} | \sF_{t_0} ] = M_{t_0}^{\psi,j} . \]
But $M^{\psi,j}_{t \wedge T_{0}^{t_0} \wedge S_n} = M^{\psi,j}_{t \wedge S_n} \mathbb{I}_{ \{ t \wedge S_n < T_{0}^{t_0} \} } + M^{\psi,j}_{ T_0^{t_0}} \mathbb{I}_{ \{ t \wedge S_n \geq T_{0}^{t_0} \} }$ and on $(t \wedge S_n \geq T^{t_0}_0)$
\begin{equation}
    \begin{split}
      &\EE \left[  M^{\psi,j}_{t \wedge S_n} \mathbb{I}_{ \{ t \wedge S_n \geq T_{0}^{t_0} \} } | \sF_{T^{t_0}_0} \right] = \\
      &\hspace{+2cm}\EE \left[ \left. \frac{1}{\sqrt{1+t \wedge S_n}}\psi \left( \frac{R_{t \wedge S_n}}{\sqrt{1+t\wedge S_n}} \right) \right| R_{T^{t_0}_0} = 0 \right] \EE[ \mathbb{I}_{\{\Theta_t = \theta_j\}} - p_j |  R_{T_{0}^{t_0}}=0] = 0,  
    \end{split}
\end{equation}
by the independence of the radial and angular parts of the process. 
Then, taking conditional expectations on $\sF_{t_0}$, $\EE \left[  M^{\psi,j}_{t \wedge S_n} \mathbb{I}_{ \{ t \wedge S_n \geq T_{0}^{t_0} \} } | \sF_{t_0} \right] = 0$. Then
\begin{eqnarray*}
    \EE[ M^{\psi,j}_{t \wedge S_n} | \sF_{t_0} ] & = &
       \EE[ M^{\psi,j}_{t \wedge S_n} \mathbb{I}_{ \{ t \wedge S_n < T_{0}^{t_0} \} } | \sF_{t_0}] + \EE[ M^{\psi,j}_{t \wedge S_n} \mathbb{I}_{ \{ t \wedge S_n\geq T_{0}^{t_0} \} } | \sF_{t_0} ] \\
       & = & \EE[ M^{\psi,j}_{t \wedge T_{0}^{t_0} \wedge S_n} | \sF_{t_0} ]  = M^{\psi,j}_{t_0} . 
\end{eqnarray*}
Hence $(S_n)_{n \geq 1}$ is a reducing sequence for $M^{\psi,j}$.

(iii) follows immediately from (i) and (ii).

\end{proof}

Let $\sigma = \sigma^{t_0}_d = \min_{j \neq i} \{ \tau^{t_0}_j \} = \min_{j \neq i} \{ u \geq t_0 : \Theta_u = \theta_j, R_u \geq d_j \sqrt{1+u} \}$.

\begin{corollary}
\label{cor:M*}
Suppose $C>0$, $D_i > 0$ and $\sum_{j=1}^N p_j D_j =0$. Then $(M^{C,D}_{t\wedge \sigma})_{t_0 \leq t \leq 1}$ is a supermartingale.
Moreover, for any $b \geq 0$, $(M^{C,D}_{t\wedge \tau^b_{i} \wedge\sigma})_{t_0 \leq t \leq 1}$ is a martingale where $\tau^b_i = \inf\{ u \geq t_0 : \Theta_u = \theta_i, R_u \geq b \sqrt{1+u} \}$.
\end{corollary}

\begin{proof}
Since $C>0$ and $D_i>0$, $M^{C,D}_{t \wedge \sigma}$ is bounded below for $t \leq \sigma$. Hence $M$ is a local martingale which is bounded below and hence a supermartingale.

Further, $(M^{C,D}_{t\wedge \tau^b_{i} \wedge\sigma})_{t_0 \leq t \leq 1}$ is bounded and hence a martingale.
\end{proof}

\begin{assumption}\label{ass:spidersupermartingale}
$\sG_i$ is such that $-\mathcal{G}_i(z)-z\mathcal{G}'_i(z)+ \mathcal{G}''_i(z)\leq 0$ on $(b^*,\infty)$.
\end{assumption}

\begin{lemma}
\label{lem:MgeqGi}
Suppose Assumption~\ref{ass:spidersupermartingale} holds. Then $\sG_i(z) \leq A(b^*) \phi(z) + B_i(b^*) \psi(z)$ for all $z \in (0,\infty)$.
\end{lemma}

\begin{proof}
The result for $z \in (0,b^*]$ has already been proved in Corollary~\ref{cor:valuebddpayoffs}. So suppose $z>b^*>0$. Let $h(z) = A(b^*) \phi(z) + B_i(b^*) \psi(z)$. Note that we have $\sG_i(b^*) = h(b^*)$ and $\sG'_i(b^*) = h'(b^*,\theta_i)$ by Lemma~\ref{lem:smoothfit}.

Fix $0<a<b^*<c$ and for $Z = (Z_t)_{t\geq t_0}$ with $Z_{t_0}=b^*$ and $Z_t = \frac{R_t}{\sqrt{1+t}}$ let $H= \inf \{ u \geq t_0 : Z_u \notin (a,c) \}$. Note that $\PP(Z_H=a) = 1 - \PP(Z_H=c) \in (0,1)$.

Since $h$ solves $h'' - zh' -h =0$ we have that $N^h=(N^h_{t})_{t \geq t_0}$ given by $N^h_t = \frac{1}{\sqrt{1+t \wedge H}}h(Z_{t\wedge H})$ is a martingale.

Define $g$ by $g = h$ on $(0,b^*)$ and $g = \sG_i$ on $[b^*,\infty)$. Then $g$ is $C^1$, in particular it is continuously differentiable at $b^*$. Then $N^g=(N^g_{t \wedge H})_{t \geq t_0}$ given by $N^g_t = \frac{1}{\sqrt{1+t \wedge H}}g(Z_{t\wedge H})$ is a bounded supermartingale:
\begin{equation*}
    \begin{split}
     dN^g_t &= \mathbb{I}_{\{t \leq H\}} \bigg( (h'(Z_t) \mathbb{I}_{\{Z_t < b^*\}} + g'(Z_t) \mathbb{I}_{\{Z_t \geq b^*\}}) dB_t  \\
     &\hspace{+4cm}+ \frac{1}{2}(1+t)^{-\frac{3}{2}}( \sG_i''(Z_t) - Z_t \sG'_i(Z_t) - \sG_i(Z_t) ) \mathbb{I}_{\{ Z_t \geq b^*\}} dt \bigg) \\
     &\hspace{+7cm}\leq (h'(Z_t) \mathbb{I}_{\{Z_t < b^*\}} + g'(Z_t) \mathbb{I}_{\{Z_t \geq b^*\}}) dB_t.   
    \end{split}
\end{equation*}

Then, 
by the Optional Sampling Theorem, 
$$\EE\left[\frac{1}{\sqrt{1+H}}g(Z_H) \bigg| \sF_{t_0}\right] \leq \frac{1}{\sqrt{1+t_0}} g(b^*) = \frac{1}{\sqrt{1+t_0}} h(b^*) = \EE\left[\frac{1}{\sqrt{1+H}}h(Z_H) \bigg| \sF_{t_0}\right]$$
and so
\begin{equation*}
    \begin{split}
      &g(a) \EE\left[ \frac{1}{\sqrt{1+H}} \mathbb{I}_{\{Z_{H}=a\}} \bigg| \sF_{t_0}\right] + g(c) \EE\left[ \frac{1}{\sqrt{1+H}} \mathbb{I}_{\{Z_{H}=c\}} \bigg| \sF_{t_0}\right] \\
      &\hspace{+5cm}\leq h(a) \EE\left[ \frac{1}{\sqrt{1+H}} \mathbb{I}_{\{Z_{H}=a\}} \bigg| \sF_{t_0}\right] + h(c) \EE\left[ \frac{1}{\sqrt{1+H}} \mathbb{I}_{\{Z_{H}=c\}} \bigg| \sF_{t_0}\right]. 
    \end{split}
\end{equation*}
Since $\EE\left[ \frac{1}{\sqrt{1+H}} \mathbb{I}_{\{Z_{H}=c\}} \bigg| \sF_{t_0}\right]>0$ and $g(a)=h(a)$ we conclude that $\sG_i(c) =g(c) \leq h(c) = A(b^*)\phi(c) + B_i(b^*) \psi(c)$.
\end{proof}

\begin{theorem}[Verification theorem]
\label{thm:spiderVerification}
    Suppose that for $j \neq i$, Player $j$ uses the threshold strategy $\tau_j :=\inf\{s\geq 0: R_{s}\geq  d_j \sqrt{1+s}, \Theta_{s} = \theta_j \}$. Under Assumptions~\ref{ass:Gamma_N} and \ref{ass:spidersupermartingale}, the value function $\overline{V} = \overline{V}(t,r,\theta)$ defined in \eqref{Valuefunc_P1_E} coincides with the value function $V_{*} = V_{*}(t,r,\theta)$ given by \eqref{Expression_VE_star}. Moreover, the stopping time
    \begin{equation*}
        \tau^{*}:=\inf\{s\geq 0: R_{s}\geq b^*\sqrt{1+s}, \Theta_{s} = \theta_i\},
    \end{equation*}
    where $b^*>0$ is the maximizer of $\Gamma \equiv B_i$, is optimal for Player $i$.
\end{theorem}

\begin{proof}
  Fix $t_{0}\in[0,\infty)$. Note that for $j \neq i$, on $(\Theta_{t_0} = \theta_j, R_{t_0} \geq d_j \sqrt{1+t_0})$ we have $\tau_j^{t_0} = t_0$ and $\overline{V}(t,r,\theta) = \GG_{j}(t_0, r, \theta) = V_{*}(t_0, r, \theta)$.
  For other values of $(R_{t_0},\Theta_{t_0})$ we show that $V_{*}(t_{0},r,\theta)=\overline{V}(t_{0},r,\theta)$ by first proving $V_{*}(t_{0},r,\theta)\geq \overline{V}(t_{0},r,\theta)$, and then proving the reverse inequality.

Suppose that $(R_{t_0},\Theta_{t_0})$ is such that either $\Theta_{t_{0}} = \theta_j$ for $j \neq i$ and $R_{t_0} < d_j \sqrt{1+ t_0}$ or $\Theta_{t_{0}} = \theta_i$
or $R_{t_0}=0$.
Let $M^* = (M^{A(b^*), (B_j(b^*))_{1 \leq j \leq N}}_t)_{ t\geq t_0}$.
Then
\[ 
M^*_t  =  \frac{1}{\sqrt{1+t}} A(b^*) \phi\left( \frac{R_t}{\sqrt{1+t}} \right)
 + \sum_{j=1}^N B_j (b^*) \frac{1}{\sqrt{1+t}} \psi\left( \frac{R_t}{\sqrt{1+t}} \right) \mathbb{I}_{ \{ \Theta_t = \theta_j \} }
\] 
and $(M^*_{t \wedge \sigma})_{t_0 \leq t \leq 1}$ is a supermartingale by Corollary~\ref{cor:M*}.

For any $\tau \in \sT^{i,t_0}$, 
on the set $t_0 \leq \sigma$ we have $M^*_{t_0} \geq \EE[ M^*_{\tau \wedge \sigma} | \sF_{t_0} ]$. Then, using that 
$Z_{\tau^{t_0}_j} = d_j$ on $\tau^{t_0}_j = \sigma$,
\begin{eqnarray*}
M^*_{\tau \wedge \sigma} & = &
\mathbb{I}_{\{ \tau < \sigma\} } \frac{1}{\sqrt{1+ \tau}} \left( A(b^*) \phi ( Z_\tau ) + B_i(b^*) \psi(Z_\tau) \right) \\
  & &  + \sum_{j \neq i} \mathbb{I}_{\{\tau^{t_0}_j = \sigma < \tau\}} \frac{1}{\sqrt{1+ \tau^{t_0}_j}} ( A(d_j) \phi(Z_{\tau^{t_0}_j}) +  B_j (d_j) \psi( Z_{\tau^{t_0}_j} ) ) \\
 & \geq &
\mathbb{I}_{\{ \tau < \sigma \}} \frac{1}{\sqrt{1+ \tau}} \sG_i(Z_\tau) + \sum_{j \neq i} \mathbb{I}_{\{\tau^{t_0}_j = \sigma < \tau\}}  \frac{1}{\sqrt{1+ \sigma}} \sG_j(d_j) \\
& = & \mathbb{I}_{ \{ \tau < \sigma \} } \GG_i(\tau, R_{\tau},\Theta_\tau) + \sum_{j=1, j \neq i}^N \mathbb{I}_{ \{ \tau^{t_0}_j = \sigma < \tau \} }  \GG_j(\tau^{t_0}_j,R_{\tau^{t_0}_j},\theta_j) \\
\end{eqnarray*}
where we use Lemma~\ref{lem:MgeqGi} for the inequality. 
Then, on $t_0 < \sigma$,
\begin{eqnarray*} \overline{V}((t_0,R_{t_0}, \Theta_{t_0}) & = & \sup_{\tau \in \sT_i^{t_0}} \EE \left[ \left. \mathbb{I}_{ \{ \tau < \sigma \} } \GG_i(\tau,R_{\tau},\theta_i) + \sum_{j=1, j \neq i}^N \mathbb{I}_{ \{ \tau^{t_0}_j = \sigma  < \tau \} } \GG_j(\tau^{t_0}_j, R_{\tau^{t_0}_j},\theta_j) \right| \sF_{t_0} \right] \\
& \leq & M^*_{t_0}  = V_{*}(t_0,R_{t_0}, \Theta_{t_0}).
\end{eqnarray*}

Now we aim to show the reverse inequality i.e. $\overline{V}(t_{0},r,\theta)\geq V_{*}(t_{0},r,\theta)$.
This will follow if $M^*_{t \wedge \tau^* \wedge \sigma}$ is a martingale where $\tau^* = \inf \{ u \geq t_0: R_u \geq b^*\sqrt{1+u}, \Theta_u =\theta_i \}$ and if $A(b^*) \phi(b^*) + B_i(b^*) \psi(b^*) = \sG(b^*)$. For $t_0 \leq t \leq \sigma_d \wedge \tau^*$ the martingale property follows from the second part of Corollary~\ref{cor:M*}, and the identity at $b^*$ follows from \eqref{eq:ABvaluematching} applied at $b^*$.

\end{proof}

\subsection{Subgame-Perfect Nash Equilibria in the $N$-Player Optimal Stopping Game}
\label{ssec:spiderNE}

In Section~\ref{ssec:gameformulation} we formulated the multi-player version of the optimal stopping game considered in Section 2. Then, in Section~\ref{ssec:gameosp}, fixing all other player strategies, we found that the optimal response for Player $i$ is the global maximizer of $\Gamma_{i,b_{-i}}$ is given by
$\Gamma_{i,b_{-i}}:(0,\infty) \to \RR$ and
\begin{equation}
\label{eq:spiderGamma} \Gamma_{i,b_{-i}}(b_i) = \frac{ \sE_{-i} \sG^i_i(b_i) - \sD_{-i} \phi(b_i)}{\sE_{-i} \psi(b_i) + p_i \phi(b_i)}.  \end{equation}
Here $b_i$ denotes the threshold Player $i$ chooses and $b_{-i}$ denotes the thresholds chosen by all other players. To simplify the notation in the following sections, we write $\sG_i:=\sG_i^i$.
Note that $i$ is an arbitrary index from $\{1,...,N\}$. In this section, our goal is to establish sufficient conditions for the existence of a subgame-perfect Nash equilibrium in this $N$-player game.

Let $\BB$ denote the set of $N$-tuples with elements in $(0,\infty)$ and let $\BB_{-i}$ be the set of elements of $\BB$ with the $i^{th}$ element omitted. In particular,
\begin{eqnarray*}
\BB & = & \{ (b_1, \ldots, b_k, \ldots, b_N) : 0 <  b_k < \infty, 1 \leq k \leq N \} \\
\BB_{-i} & = & \{ (b_1, \ldots, b_{i-1}, b_{i+1}, \ldots, b_N) : 0 <  b_k < \infty, k \neq i, 1 \leq k \leq N \}    
\end{eqnarray*}
where $\BB_{-1} = \{ (b_2, \ldots, b_N) : 0 <  b_k < \infty, 2 \leq k \leq N \}$ and $\BB_{-N} = \{ (b_1, \ldots, b_{N-1}) : 0 <  b_k < \infty, 1 \leq k \leq N-1 \}$.

Let $h_i$ be given by $h_i(z)=\sG_i'(z) - z \sG_i(z)$.

\begin{assumption}\label{ASS_supermatingale_Nplayer}
   Suppose that for any $i,j\in\{1,...,N\}$, 
   \begin{enumerate}
       \item $\mathcal{G}'_i(z)>0$, and for any $j\neq i$, $\mathcal{G}_j^{i}(z)<0$ for all $z\in(0,\infty)$.
       \item There exists a solution $\hat{b}_i\in(0,\infty)$ to $h_i(\cdot)=0$. Further, we assume that $h_i$ is unimodal on $(0,\infty)$, and has a unique maximiser $\hat{z}_i \in (0,\hat{b}_i)$. 
       \item  Given any $b_{-i}\in \BB_{-i}$, we have that both $\lim_{z\to0}\Gamma_{i,b_{-i}}(z)$ and $\lim_{z\to0}h_i(z)$ exist in $[-\infty,\infty)$, and satisfy $\lim_{z\to0}\Gamma_{i,b_{-i}}(z)<\lim_{z\to0}h_i(z)$.
   \end{enumerate}
   
\end{assumption}

\begin{theorem}\label{Theorem_multi_NE}
    Under Assumption \ref{ASS_supermatingale_Nplayer}  there exists a subgame-perfect Nash equilibrium for the multi-player optimal stopping game introduced in Section~\ref{ssec:gameformulation}.
\end{theorem}
\begin{proof}
    Assumption \ref{ASS_supermatingale_Nplayer} implies that $\sG_i'(\hat{b}_i)>0$, and hence, since $h_i(\hat{b}_i)=\sG_i'(\hat{b}_i) - z \sG_i(\hat{b}_i)=0$, we have $\sG_i(\hat{b}_i)>0$. Assumption \ref{ASS_supermatingale_Nplayer} also implies that $\sG^i_j(\hat{b}_i)<0$ and hence $\sD_{-i}<0$. Thus, we conclude from \eqref{eq:spiderGamma} that $\Gamma_{i,b_{-i}}(\hat{b}_{i})>0$.

    Define $\Upsilon_i(z)=\Gamma_{i,b_{-i}}(z)- h_i(z)$. From the fact that $h_i(\hat{b}_i)=0$ we have that $\Upsilon_i(\hat{b}_i)>0$. Further, since by assumption $\lim_{z\to0}\Gamma_{i,b_{-i}}(z) < \lim_{z\to0}h_i(z)$ we have that $\lim_{z\to0}\Upsilon_i(z)<0$. Then, by the Intermediate Value Theorem, we know that there exists a point $b^*_i\in (0,\hat{b}_i)$ such that $\Upsilon_i(b^*_i)=0$. Moreover. with the same argument given in Proposition \ref{Suff_conditions_ass5}, 
    and using
    \[ \Gamma_{i,b_{-i}}'(z) = \frac{\sE_{-i}}{p_{i} \phi(z) + \sE_{-i} \psi(z)} \left( h_i(z) - \Gamma_{i,b_{-i}}(z) \right) = - \frac{\sE_{-i}}{p_{i} \phi(z) + \sE_{-i} \psi(z)} \Upsilon_i(z) \]
    we have that $\Gamma'_{i,b_{-i}}(z)>0$ for $z\in(0,b ^*_i)$ and $\Gamma'_{i,b_{-i}}(z)<0$ for $z\in(b^*_i,\infty)$. Moreover,
    $b^*_i\in(\hat{z}_i,\hat{b}_i)$ where $\hat{z}_i$ is the unique maximiser of $h_i$ on $(0,\hat{b})$. Hence, we have that $\Upsilon_i(\hat{z}_i)<0$ and $\Upsilon_i(\hat{b}_i)>0$.

    For the rest of the proof we re-express $\Gamma_{i,b_{-i}}$ as a multivariate function $\Gamma_i:(0,\infty)^{N}\mapsto \mathbb{R}$ given by
    \begin{equation*}
      \begin{split}
            \Gamma_{i}(b_1,...,b_N) = \frac{ \sum_{j \neq i, j=1}^{N} \frac{p_{j}\phi(b_{j})}{\psi(b_j)} \sG_i(b_i) - \sum_{j \neq i,j=1}^{N} \frac{p_{j}\sG^i_{j}(b_{j})}{\psi(b_j)} \phi(b_i)}{\sum_{j \neq i, j=1}^{N} \frac{p_{j}\phi(b_{j})}{\psi(b_j)} \psi(b_i) + p_i \phi(b_i)}. 
        \end{split}  
    \end{equation*}

    Define $\DD=\{\bar{b}=(b_1,...,b_N)\in\mathbb{R}^N: \forall i\in\{1,...,N\}, \hat{z}_i\leq b_i\leq \hat{b}_i\} \subset \BB$. Moreover, let 
    \begin{equation*}
        \DD_{i}^{-}:=\{\bar{b}\in \DD: b_i=\hat{z}_i\}, \hspace{+0.5cm} \DD_{i}^{+}:=\{\bar{b}\in \DD: b_i=\hat{b}_i\}.
    \end{equation*}
    For each $i\in\{1,...,N\}$ define $\tilde{\Upsilon}_i:\DD\mapsto\mathbb{R}$, by $\tilde{\Upsilon}_i(b_1,...,b_N):=\Gamma_{i}(b_1,...,b_N)-h_i(b_i)$. 
    What we have proved so far can be summarized as: for any $\bar{b}\in\DD_{i}^{-}$, $\tilde{\Upsilon}_i(\bar{b})<0$;  for any $\bar{b}\in\DD_{i}^{+}$, $\tilde{\Upsilon}_i(\Bar{b})>0$. Let $\tilde{\Upsilon}:\DD\mapsto\mathbb{R}^N$ be given by $\tilde{\Upsilon}(\bar{b})=(\tilde{\Upsilon}_1(\Bar{b}),...,\tilde{\Upsilon}_N(\Bar{b}))$; then $\tilde{\Upsilon}$ is continuous and $\DD$ is a compact set.  Hence, applying Miranda's Theorem \cite{Miranda1940}, we know that there exists $\bar{b}^*\in \DD$ such that $\tilde{\Upsilon}(\bar{b}^*)= (\tilde{\Upsilon}_1(\bar{b}^*),...,\tilde{\Upsilon}_N(\bar{b}^*))=(0,...,0)$.
    In particular, for each $k$, $\tilde{\Upsilon}_k(\bar{b}^*)=0$ and $\bar{b}^*_k$ is the maximiser of $\Gamma_{k,\bar{b}^*_{-k}}(b)$.

It remains to show that $\bar{b}^*$ can be used to define a subgame-perfect Nash equilibrium. Let $\tau^*_k = \inf \{u\geq 0: R_u \geq b^*_k\sqrt{1+u}, \Theta_u = \theta_k \}.$  With the same argument given in Proposition \ref{Suff_conditions_ass5}, Assumption \ref{ASS_supermatingale_Nplayer} implies Assumption \ref{ass:Gamma_N} and \ref{ass:spidersupermartingale}.
By Theorem~\ref{thm:spiderVerification}, conditional on the fact that for $j \neq k$ Player $j$ uses the stopping rule $\tau^*_j = \inf \{u\geq 0: R_u \geq b^*_j\sqrt{1+u}, \Theta_u = \theta_j \}$ we have that $\tau^*_k$ is optimal for Player $k$. Hence, $(\tau^*_j)_{1 \leq j \leq N}$ is a subgame-perfect Nash equilibrium.
\end{proof}

\subsection{An Example with linear payoffs}
Now we assume that the payoff functions are linear. More specifically, we let $\mathcal{G}_{j}^{i}(z)=Q^{i}_{j}z-\alpha_i\delta_{ij}$ where $(Q_{j}^{i})_{1\leq i,j\leq N}$ is a constant matrix in $\mathbb{R}^{N}\times\mathbb{R}^{N}$ with $Q^{i}_{i}>0$ and $Q^{i}_{j}<0$ when $j\neq i$. Moreover, $\delta_{ij}=1$ if $i=j$ and zero otherwise, and $\alpha=(\alpha_i)_{1\leq i\leq N}$ is a constant vector in $(0,\infty)^{N}$. Our goal is to establish sufficient conditions for the existence of a subgame-perfect Nash equilibrium when the payoffs are linear. 

\begin{theorem}
    There exists a subgame-perfect Nash equilibrium in the multi-player version of the optimal stopping game with linear payoffs, if for every $i\in\{1,...,N\}$ 
    \begin{equation}
    \label{eq:spidersufficient}
        p_iQ_i^i+ \sum_{j\neq i,j=1}^N p_jQ_j^i\geq0.
    \end{equation}
\end{theorem}

\begin{proof}
     This is a direct application of Theorem \ref{Theorem_multi_NE}. We need to show that Assumption~\ref{ASS_supermatingale_Nplayer}
     is satisfied.
     
     It is immediate that for any $i,j\in\{1,...,N\}$, $\mathcal{G}'_i(z)>0$ and $\mathcal{G}^i_j(z)<0$ for all $z\in(0,\infty)$. Moreover, with the specific linear payoffs of this example we have that $h_i(z)=-Q_i^i z^2+\alpha_i z+Q_i^i$. Therefore, $h_i$ has a unique root $\hat{b}_i= \hat{b}_{i}(Q_i^i, \alpha_i)=\frac{\alpha_i+\sqrt{\alpha_i^2 +4Q_i^i}}{2Q_i^i}$. Furthermore, $h_i$ is unimodal and has a unique maximiser $\hat{z}_i=\frac{\alpha_i}{2Q_i^i}$ on the interval $(0,\hat{b}_i]$.
    Given any $b_{-i}\in\DD_i$, we have that 
    $\Gamma_{i, b_{-i}}$ is given by
     \begin{equation*}
        \begin{split}
            \Gamma_{i, b_{-i}}(z) = \frac{ \sE_{-i} (Q_i^i z-\alpha_i) - \sD_{-i} \phi(z)}{\sE_{-i} \psi(z) + p_i \phi(z)}.
        \end{split}
    \end{equation*}
    Note that $\lim_{z\to0}\Gamma_{i,b_{-i}}(z)=\frac{-\sE_{-i}\alpha_i-\sD_{-i}}{p_i}$ and $\lim_{z\to0}h_i(z)=Q_i^i$. It only remains to find sufficient conditions such that $\lim_{z\to0}\Gamma_{i,b_{-i}}(z)<\lim_{z\to0}h_i(z)$: then Assumption \ref{ASS_supermatingale_Nplayer} is satisfied. It is equivalent to find sufficient conditions for the inequality $Q_i^i p_i+\sE_{-i}\alpha_i+\sD_{-i}>0$ to hold.
    Substituting expressions for $\sE_{-i}$ and $\sD_{-i}$, we want a sufficient condition for
    \begin{equation}\label{Multi_ineq_NE}
        p_i Q_i^i+ \alpha_i \sum_{j \neq i, j=1}^{N} \frac{p_{j}\phi(b_{j})}{\psi(b_j)} + \sum_{j \neq i,j=1}^{N} \frac{p_{j}Q_j^i b_j}{\psi(b_j)} >0,
    \end{equation}
    where $Q_i^i>0$ and $Q_j^i<0$ when $j\neq i$. Since $\frac{x}{\psi(x)}$ on $(0,\infty)$ we have that $\frac{Q_j^i b_j}{\psi(b_j)} > Q^i_j$ and then the left-hand-side of \eqref{Multi_ineq_NE} is bigger than the left hand side of \eqref{eq:spidersufficient}. Hence, \eqref{eq:spidersufficient} is sufficient for Assumption~\ref{ASS_supermatingale_Nplayer}. 
\end{proof}

\printbibliography
\end{document}